\documentclass[dvips,11pt]{article}

  \usepackage[pdftex]{graphicx}
  \usepackage{url}
  \usepackage{amsthm}
  \usepackage{verbatim}
  \usepackage{amsmath}
  \usepackage{amssymb}

  \theoremstyle{plain}
  \newtheorem{theorem}{Theorem}[section]
  \newtheorem{lemma}[theorem]{Lemma}
  \newtheorem{corollary}[theorem]{Corollary}

  \newtheorem{fact}[theorem]{Fact}
  
  \newtheorem{claim}[theorem]{Claim}
  \newtheorem{definition}[theorem]{Definition}
  \newtheorem{example}[theorem]{Example}
  \newtheorem{conjecture}[theorem]{Conjecture}

  \theoremstyle{remark}

  \newcounter{problempart}
  \newenvironment{parts}{
  \begin{list}{(\alph{problempart})}{
  \setlength{\leftmargin}{0in}
  \setlength{\itemsep}{\medskipamount}
  \setlength{\itemindent}{.3in}
  \setlength{\parsep}{\medskipamount}
  \setlength{\parskip}{\bigskipamount}
  \usecounter{problempart}
  }
  }{
  \end{list}
  }  
  
\begin{document}
  
  
  \title{Even cycles in dense graphs}
  \author{Neal Bushaw\thanks{Dept. of Math. \& Appl. Math., Virginia Commonwealth University}, Andrzej Czygrinow\thanks{School of Mathematics and Statistical Sciences, Arizona State University}, Jangwon Yie\thanks{School of Mathematics and Statistical Sciences, Arizona State University}}
  \date{\today}
  

  \maketitle
  \begin{abstract}
  We show that for $\alpha>0$ there is $n_0$ such that if $G$ is a graph on $n\geq n_0$ vertices such that $\alpha n< \delta(G)< (n-1)/2$, then for every $n_1+n_2+\cdots +n_l= \delta(G)$, $G$ contains a disjoint union of $C_{2n_1},C_{2n_2}, \dots, C_{2n_l}$ unless $G$ has a very specific structure. This is a strong form of a conjecture of Faudree, Gould, Jacobson, and Magnant for large dense graphs; it a generalization of a well known conjecture of Erd\H{o}s and Faudree (since solved by Wang) as well as a special case of El-Zahar's conjecture.
  \end{abstract}
  
  \section{Introduction}
  Determining the structure of cycles in graphs is a problem of fundamental interest in graph theory; this thread traces through numerous subareas in structural and extremal graph theory. Throughout this paper all graphs are simple and undirected; we use standard notation wherever possible.
  
  For a graph $G$ we use $c(G)$ to denote the circumference of $G$, and $oc(G)$ ($ec(G)$) to denote the length of the longest odd (even) cycle in $G$. If $G$ is a graph of minimum degree $d$, then $c(G)\geq d$ which is the best possible. However, additional assumptions on the connectivity of $G$ usually lead to better bounds for $c(G)$ (or $ec(G)$ and $oc(G)$). For example, Dirac's theorem states that if $G$ is a 2-connected graph on $n$ vertices, then $c(G)\geq \min\{n ,2\delta(G)\}$. Voss and Zuluaga \cite {VZ} proved the corresponding results for $ec(G)$ and $oc(G)$.
  \begin{theorem}[Voss and Zuluaga]
  Let $G$ be a 2-connected graph on $n\geq 2\delta(G)$ vertices. Then $ec(G)\geq 2\delta(G)$ and $oc(G)\geq 2\delta(G)-1$.
  \end{theorem}
  Dirac's Theorem gave birth to a large body of research centered around determining the length of the longest cycle in a graph satisfying certain conditions; we direct the interested reader to, e.g., \cite{MGT}. Indeed, one could even search for graphs which contain cycles of all possible lengths. Such graphs are called pancyclic, and they, too, are well studied (see, e.g., \cite{Bondy, BondySim, BondyVince, BFG}). Bondy observed that in many cases a minimum degree which implies the existence of a spanning cycle also implies that the graph is pancyclic. For example, it follows from the result in \cite{Bondy} that if $G$ is a graph on $n$ vertices with minimum degree at least $n/2$ then $G$ is either pancyclic or $G=K_{n/2,n/2}$.  It's natural to ask if analogous statements are true for graphs with smaller minimum degree. In \cite{GHS}, Gould et. al. proved the following result.
  \begin{theorem}[Gould, Haxell and Scott]
  For every $\alpha>0$ there is $K$ such that if $G$ is graph on $n>45 K\alpha^{-4}$  vertices with $\delta(G) \geq \alpha n$, then $G$ contains cycles of every  even length from $[4, ec(G)-K]$ and every odd length from $[K, oc(G)-K]$.
  \end{theorem}
  Nikiforov and Shelp (\cite{NS}) proved that if $G$ is a graph on $n$
  vertices with $\delta(G)\geq \alpha n$, then $G$ contains cycles of every even lengths from $[4, \delta(G)+1]$ as well as cycles of odd lengths from $[2k-1, \delta(G)+1]$, where $k=\lceil 1/\alpha \rceil$, unless $G$ is one of several explicit counterexamples.
  
  One of the motivations for our work is the following conjecture of Faudree, Gould, Jacobson and Magnant \cite{FGJM} on cycle spectra. Let $S_e = \{|C| : C \text{ is an even cycle contained in } G \}$ and $S_o = \{|C| : C \text{ is an odd cycle contained in } G \}$.
  \begin{conjecture}\label{main conjecture}
  Let $d\geq 3$. If $G$ is 2-connected graph  on $n\geq 2d$ vertices such that $\delta(G) = d \geq 3$ then $|S_e| \geq d - 1$, and, if, in addition $G$ is not bipartite, then $|S_o| \geq d$.
  \end{conjecture}
  In \cite{FGJM}, Conjecture \ref{main conjecture} was confirmed for $d=3$. In addition, many related results were proved by Liu and Ma in \cite{LiuMa}. For example, they proved that if $G$ is a bipartite graph such that every vertex but one has a degree at least $k+1$, then $G$ contains $C_1, ... C_k$ where $3 \leq |C_1|$ and for all $i \in [k-1]$, $|C_{i+1}| - |C_{i}| = 2$. For general graphs, they showed that if the minimum degree of a graph $G$ is at least $k+1$, then $G$ contains $\lfloor k/2 \rfloor$ cycles with consecutive even lengths and if $G$ is 2-connected and non-bipartite, then $G$ contains $\lfloor k/2 \rfloor$ cycles with consecutive odd lengths.
  
  Another line of research which motivates our work comes from problems on 2-factors. Erd\H{o}s and Faudree \cite{EF} conjectured that every graph on $4n$ vertices with minimum degree at least $2n$ contains a $2$-factor consisting of $\frac{n}{4}$ copies of $C_4$, cycle on four vertices. This was proved by Wang in \cite{Wang}.
  A special case of El-Zahar's conjecture states that any graph $G$ on $2n$ vertices with minimum degree at least $n$ contains any 2-factor consisting of even cycles $C_{2n_1}, \dots, C_{2n_l}$ such that $n = \sum n_i$.
  It's natural to ask if analogous statements can be proved in the case when the minimum degree of $G$ is smaller. As we will show, this is true to some extent.  We will prove that for almost all values of $n_1, \dots, n_l$ such that $\sum n_i = \delta(G)$, $G$ indeed contains the union of disjoint cycles $C_{2n_1}, C_{2n_2}, \dots, C_{2n_l}$. There are, however, two obstructions -- of which one is well-known -- when $G$ is a graph on $n$ vertices with minimum degree satisfying $\alpha n <\delta(G) < (n-1)/2$.
  \begin{example}\label{ex1}
  Let $l \geq 2$ , $q \geq 4$ be even. We first construct graph $H$ on $l(q -2)+3$ vertices as follows.
  Let $V_1, \dots, V_l$ be disjoint sets each of size $q-2$ such that $H[V_i]= K_{q-2}$ and let $u_1,u_2,u_3$ be three distinct vertices and let $vu_i\in E(H)$ for every $v \in V(H) \setminus \{u_1,u_2,u_3\}$
  and every $i=1,2,3$. Finally let $G_k$ be obtained from $H$ by adding exactly $k$ out of the three possible edges between vertices from $\{u_1,u_2,u_3\}$.
  Then $\kappa(G_k)=3$, $\delta(G_k)= q$ but $G_k$ does not contain two disjoint copies of $C_q$. Indeed, any copy of $C_q$ in $G_k$ contains at least two vertices from $\{u_1,u_2,u_3\}$.
  \end{example}

  In addition to the obstruction from Example \ref{ex1}, another one arises when $G$ is very close to being a complete bipartite graph.
  \begin{example}\label{ex2}
  Let $q=2k$ for some $k\in \mathbb{Z}^+$ and let $U, V$ be disjoint and such that $|U|=q-1, |V|= n- q+1$ with $n-q+1$ even. Let $G[U,V] =K_{q-1,n-q+1}$, $G[U] \subset K_{q-1}$ and $G[V]$ is a perfect matching. Then $G$ is a 2-connected graph on $n$ vertices with $\delta(G)=q$ which doesn't have $q/2$ disjoint copies of $C_4$. Indeed, if there are $q/2$ disjoint copies of $C_4$, then at least one must contain at least three vertices from $V$ which is not possible.
  \end{example}
 The main result of the paper is the following.
  
  \begin{theorem}\label{thm:Mainthm}
  For every $0< \alpha< \frac{1}{2}$, there is a natural number $N=N(\alpha)$ such that the following holds.
  For any $n_1,..., n_l\in \mathbb{Z}^+$ such that $\sum_{i=1}^{l}n_i = \delta(G)$ and $n_i \geq 2$ for all $i \in [l]$, every $2$-connected graph $G$ of order $n \geq N$ and $\alpha n \leq \delta(G) < n/2-1$ contains the disjoint union of $C_{2n_1}, \dots, C_{2n_l}$, or $G$ is one of the graphs from Example \ref{ex1} and $l=2,2n_1=2n_2=\delta(G)$, or $G$ is a subgraph of the graph from Example \ref{ex2} and $n_i=2$ for every $i$.
  \end{theorem}	
  In the case when $\delta(G)\geq n/2-1$ additional counterexamples appear when $n_i=2$ for every $i$; these will be characterized in the proof.
  
  As a corollary, we have the following fact which answers the question of Faudree, Gould, Magnanat, and Jacobson in the case of dense graphs.
  \begin{corollary}\label{cor:Dense-case}
  For every $0< \alpha <\frac{1}{2}$, there is a natural number $M=M(\alpha)$ such that the following holds.
  Every $2$-connected graph $G$ of order $n \geq M$ and $\sum_{v \in V(G)}d(v) \geq \alpha n^2$ contains a cycle of length $2m$ for every $m \in \{2, \dots, \delta(G)\}$.
  \end{corollary}
  In addition, the following generalization of the Erd\H{o}s-Faudree conjecture follows from Theorem \ref{thm:Mainthm}.
  \begin{corollary}
  For every $0<\alpha<\frac{1}{2}$, there is a natural number $M=M(\alpha)$ such that the following holds. Every $2$-connected graph $G$ of order $n\geq M$ and minimum degree $\delta$, such that $\alpha n\leq \delta < n/2-1$ and $\delta+ n$ is even, contains $\delta/2$ disjoint cycles on four vertices.
  \end{corollary}
  This follows immediately; since $n-\delta+1$ is odd, it is not possible to end up in Example \ref{ex2}.
  
  The proof of Theorem \ref{thm:Mainthm} uses the regularity method. The obstruction from Example \ref{ex1} appears in the proof of the non-extremal case and the obstruction from Example \ref{ex2} comes up when dealing with the extremal case.
  
  The proof is quite involved, and so we divide it into several sections to aid readability. In Section \ref{sec:reg}, we review Szemer\'edi's celebrated Regularity Lemma, as well as a special case of the well-known Blow-Up Lemma which is of particular use to us. In Section \ref{sec:prelim}, we make use of regularity and results from \cite{CK}  to find cycles of many different lengths. Following this, we consider several cases depending on (1) the structure of the reduced graph, and (2) whether or not the graph is near what we call the extremal graph. The non-extremal cases are proven in Section \ref{sec:nonextr1} and Section \ref{sec:nonextr2}, while the extremal cases follow in Section \ref{sec:extr}. Combining these gives our main result (Theorem \ref{thm:Mainthm}) for every sufficiently large graph.
  
  \section{The Regularity and Blow-Up lemmas}\label{sec:reg}
  In this section, we review concepts related to the regularity and blow-up lemmas.
  Let $G$ be a simple graph on $n$ vertices and let $U,V$ be two disjoint non-empty subsets of $V(G)$. We define the \emph{density} of $(U,V)$ as
  \begin{align*}
  d(U,V) = \frac{e(U,V)}{|U| \cdot |V|},
  \end{align*}
  where $e(U,V)=|E(U,V)|$. Further, we call the pair $(U,V)$ \emph{$\epsilon$-regular} if for every $U' \subset U$ and every $V' \subset V$ with $|U'| \geq \epsilon |U|, |V'| \geq \epsilon |V|$,
  we have
  \begin{align*}
  |d(U',V') - d(U,V)| \leq \epsilon.
  \end{align*}
  In addition, the pair $(U,V)$ is called $(\epsilon, \delta)-$\emph{super-regular} if it is both $\epsilon-$regular and furthermore for any $u\in U$ we have $|N(u) \cap V| \geq \delta |V|$, and for any $v\in V$ we have $|N(v) \cap U| \geq \delta |U|$.
  
  We continue with the definition of a regular partition.
  \begin{definition}
  A partition $\{V_0,V_1,...,V_t \}$ of $V(G)$ is called $\epsilon$-regular if the following conditions are satisfied.
  \begin{parts}
  \item[(i)] $|V_0|\leq \epsilon |V(G)|$.
  \item[(ii)] For all $i,j \in [t], |V_i| = |V_j|$.
  \item[(iii)] All but at most $\epsilon t^2$ pairs $(V_i,V_j), i,j \in [t]$ are $\epsilon-$regular.
  \end{parts}
  \end{definition}
  The Regularity Lemma of Szemer\'{e}di (\cite{Sz}) states that every graph admits an $\epsilon$-regular partition in which the number of partition classes is bounded.
  \begin{lemma}[Regularity Lemma,\cite{Sz}] \label{regularity}
  For every $\epsilon >0, m >0$ there exist $N := N(\epsilon,m)$ and $M := M(\epsilon, m)$ such that every graph on at least $N$ vertices has an $\epsilon-$regular partition $\{V_0,V_1,...,V_t\}$  such that $m \leq t \leq M$.
  \end{lemma}
  In addition to the regularity lemma, we will need a few well-known facts about $\epsilon$-regular pairs and the so-called Slicing Lemma (see, e.g., \cite{KS}), as well as the Blow-Up Lemma of Koml\'{o}s, S\'{a}rk{o}zy and Szemer\'{e}di  \cite{KSS}.
  
  In particular, we will need the so-called slicing lemma, see \cite{KS}.
  \begin{lemma}[Slicing Lemma]\label{slicing lemma}
  Let $(U,V)$ be an $\epsilon$-regular pair with density $\delta$, and for some $\lambda > \epsilon,$ let $U' \subset U, V' \subset V$ with $|U'| \geq \lambda |U|, |V'| \geq \lambda |V|$. Then $(U',V')$ is an $\epsilon'$-regular pair of density $\delta'$ where $\epsilon' = \max \{\frac{\epsilon}{\lambda}, 2 \epsilon \}$ and $\delta' \geq \delta - \epsilon$.
  \end{lemma}
  
  It is not difficult to see that an $\epsilon$-regular pair of density $\delta$ contains a large $(\epsilon', \delta')$-super-regular pair for some $\delta', \epsilon'$. 
  
  \begin{lemma} \label{regularTosuperRegular}
  Let $0<\epsilon <\delta/3 < 1/3$ and let $(U,V)$ be an $\epsilon-$regular pair with density $\delta$. Then there exist $A' \subset A$ and $B' \subset B$ with $|A'| \geq (1-\epsilon)|A|$ and $|B'| \geq (1-\epsilon)|B|$ such that $(A',B')$ is a $(2\epsilon, \delta - 3\epsilon)-$super-regular pair.
  \end{lemma}
  Let $0<\epsilon \ll \delta <1$. For an $\epsilon$-regular partition $\{V_0, V_1, \dots, V_t\}$ of $G$ we will consider the {\it reduced graph} (or {\it cluster graph}) of $G$, $R_G=R_{\epsilon, d}(V_0,V_1,\dots, V_t)$ where $V(R_G)=\{V_1, \dots, V_t\}$ and $V_iV_j \in E(R_G)$ if $(V_i,V_j)$ is
  $\epsilon$-regular with density at least $d$. When clear from the context, we will omit the subscript, writing $R$ for the cluster graph at hand.
  
  
  Finally, we conclude this section with the statement of a special case of the blow-up lemma.
  \begin{lemma}[Blow-Up Lemma, \cite{KSS}] \label{blowup}
  Given $d > 0, \Delta > 0$ and $\rho > 0$ there exists $\epsilon > 0$ and $\eta > 0$ such that the following holds. Let $S = (W_1,W_2)$ be an $(\epsilon, d)$-super-regular pair with $|W_1| =n_1$ and $|W_2| = n_2$. If $T$ is a bipartite graph with bipartition $A_1,A_2$, maximum degree at most $\Delta$, and $T$ is embeddable into the complete bipartite graph $K_{n_1,n_2}$, then it is also embeddable into $S$. Moreover, for all $\eta n_i$ sized subsets $A_i' \subset A_i$ and functions $f_i : A_i' \rightarrow \binom{W_i}{\rho n_i}$ (for $i = 1,2$), $T$ can be embedded into $S$ so that the image of each $a_i \in A_i'$ is in the set $f_i(a_i).$
  \end{lemma}
  
  \section{Preliminaries}\label{sec:prelim}
  In this section, we prove a few auxiliary facts which will be useful in the main argument. Let $V_0, V_1, \dots, V_t$ be an $\epsilon$-regular partition.
  \begin{lemma}\label{regular-graph}
  Let $\Delta\geq 1$ and let  $0<\epsilon \ll \delta \ll 1/\Delta$ be such that $10\epsilon \Delta \leq \delta$. Let $H$ be graph on $\{V_1, \dots, V_q\}$  where $|V_i|=l$ with $V_iV_j\in E(H)$ if $(V_i,V_j)$ is $\epsilon$-regular with density at least $\delta$, and assume that $H$ has maximum degree $\Delta$. Let $\epsilon' = 5\Delta \epsilon$, and $\delta' = \delta /2 $ . Then for any $i \in [t]$ there exist sets $V_i'\subset V_i$  such that $|V_i'| \geq (1-\epsilon')l$ and $(V_i',V_j')$ is $(\epsilon',\delta')$-super-regular for every $V_iV_j \in E(H)$.
  \end{lemma}
  \begin{proof}
  Note that $E(H)$ can be decomposed into $\Delta+1$ matchings and so Lemma \ref{regular-graph} follows directly from Lemmas \ref{slicing lemma} and \ref{regularTosuperRegular}.\end{proof}
  
  An $n$-ladder, denoted by $L_n$, is a balanced bipartite graph with vertex sets $A = \{a_1,..., a_n\}$ and $B = \{b_1,..., b_n\}$ such that $\{a_i,b_j\}$ is an edge if and only if $|i-j| \leq 1$. We refer to the edges $a_ib_i$ as rungs and the edges $\{a_1,b_1\}, \{a_n,b_n\}$ as the first and last rung respectively. 
  Let $L_{n_1}, L_{n_2}$ be two ladders with $n_1 \leq n_2$ and $\{a_1,b_1\} ,\{a_1',b_1'\}$ the first rungs of $L_{n_1}, L_{n_2}$, respectively.
  If there exist $a_1-a_1'$ path $P_1$ and $b_1-b_1'$ path $P_2$ such that $\mathring{P_1} \cap \mathring{P_2} = \emptyset, (\mathring{P_1} \cup \mathring{P_2}) \cap (L_1 \cup L_2) = \emptyset, |\mathring{P_1}| + |\mathring{P_2}| = 2k  $,
  then we call $L_{n_1} \cup L_{n_2} \cup P_1 \cup P_2$ an $(n_1+n_2,k)-$weak ladder. Obviously, an $n-$ladder is an $(n,0)-$weak ladder.
  
  Weak ladders are of use for finding cycles via the following  lemmas.
  \begin{lemma} \label{weak ladder criteria - 1}
  Let $2 \leq n_1 \leq \cdots \leq n_l \in \mathbb{Z}^+$ and let $n=\sum_{i=1}^l n_i$. If $G$ contains a $(n',k)-$weak ladder for some $n',k \in \mathbb{N}$ such that $n' \geq n + k$,
  then $G$ contains disjoint cycles $C_{2n_1}, C_{2n_2}, \dots, C_{2n_l}$. 
  \end{lemma}
  \begin{proof}
  If $k=0$ then our $(n',0)$-weak ladder is simply an $n'$ ladder; since $n' \geq n$ then it is trivial that $G$ contains disjoint cycles $C_{2n_1}, C_{2n_2}, \dots, C_{2n_l}$.
  Next, we assume that $k \geq 1$.
  Suppose $G$ contains an $(n',k)-$weak ladder $L$ and $L$ consists of two ladders $L_{a_1}$, $L_{a_2}$ such that $a_1 + a_2 = n'$ and disjoint paths $P,Q$ such that  $|\mathring{P_1}| + |\mathring{P_2}| = 2k$.
  Let $N = \{n_i : i \in [l]\}$ and choose $N' \subset N$ such that $\sum_{x \in N'}x < a_1$ and with $t:=a_1 - \sum_{x \in N'}x>0$ as small as possible.
  By the construction of $N'$, for any $y \in N \setminus N'$, $y > t$. If $t \leq k$, then 
  $$a_2 = n' - a_1 = n' - (t + \sum_{x \in N'}x) \geq n' - k - \sum_{x \in N'}x \geq n - \sum_{x \in N'}x = \sum_{x \in N \setminus N'}x,$$
  which implies that $L_{a_2}$ contains the necessary cycles.
  Hence we may assume that $t \geq k+1$ and so for any $y \in N \setminus N'$, $y \geq t + 1 \geq k+2$.
  If there exists $y \in N \setminus N'$ such that $y \leq k + t + 1$, 
  then the sub weak-ladder consisting of the last $t$ rungs of $L_{a_1}$, the first rung of $L_{a_2}$, and $P,Q$ contains $C_{2y}$. In addition,
  $$ n - \sum_{x \in N' \cup \{y\}} x  \leq n-(a_1-t +y) \leq (n' - k) - a_1 - 1  \leq a_2 -k- 1,$$
  so $L_{a_2 - 1}$ contains the remaining cycles.
  Otherwise, let $y = k+t+ c$ where $c \geq 2$. The sub weak-ladder consisting of the last $t$ rungs of $L_{a_1}$, first $c$ rungs of $L_{a_2}$, and $P,Q$ contains $C_{2y}$. We have
   $$ n -  \sum_{x \in N' \cup \{y\}} = n - (a_1 - t + y) \leq (n' - k) - a_1 -k - c \leq a_2 - k  - c,$$
  so $L_{a_2 - c}$ contains the remaining cycles.
  \end{proof}
  
  For the proof of Theorems \ref{non extremal main}, \ref{non-extremal two}, our plan is to seek an $(n',r)$-weak ladder such that $n' \geq \delta + r$, and then apply Lemma \ref{weak ladder criteria - 1} to obtain the desired cycles. 
  In some situations, it is not possible to obtain either an $L_{\delta(G)}$ or a large enough weak ladder to apply Lemma \ref{weak ladder criteria - 1}. For such cases, we will use the following lemma.

  \begin{lemma} \label{weak ladder of case r = 1}
  Let $2 \leq n_1 \leq \cdots \leq n_l \in \mathbb{Z}^+$ and $n=\sum_{i=1}^l n_i$. 
  If $G$ contains a $(n,1)-$weak ladder and there exists $i \in [l]$ such that $n_i > 2$
  then $G$ contains disjoint cycles $C_{2n_1}, C_{2n_2}, \dots, C_{2n_l}$.
  \end{lemma}
  \begin{proof}
  We argue by induction on $n$.
  Since $n \geq 3$, an $(n,1)$-weak ladder contains $C_{2n}$, so we may assume that $l \geq 2$ 
  (and therefore, $n \geq 5$).
  If $n = 5$ then $n_1 = 2, n_2 = 3$ and then it is easy to see that $G$ contains $C_4,C_6$.
  Now, assume for an inductive case that $n \geq 6$ and let the weak ladder contain $L_{a_1}$, $L_{a_2}$ and disjoint paths $P_1,P_2$ which connect $L_{a_1},L_{a_2}$ (and so by definition, $|\mathring{P_1}| + |\mathring{P_2}| = 2$, and $a_1 + a_2 = n, a_1 \leq a_2$).
  Note that $a_2 \geq n_1$. If $a_2 = n_1$ then $l=2$ and $a_1 = n_2$; then $L_{a_2}$ contains $C_{2n_1}$ and $L_{a_1}$ contains $C_{2n_2}$.
  Hence we may assume that $n_1 \leq a_2 - 1$.
  If $n_1 = 2$ then the first $n_1$ rungs of $L_{a_2}$ contains $C_{2n_1}$ and 
   since there exists $i \in [l] \setminus \{1\}$ such that $n_i > 2$.
  By the induction hypothesis the remaining $(n-n_1,1)$-weak ladder contains $C_{2n_2}, \dots, C_{2n_l}$.
  Hence we may assume that $n_1 > 2$, i.e, for any $i \in [l]$, $n_i > 2$.
  Since $n_1 \leq a_2 - 1$, the first $n_1$ rungs of $L_{a_2}$ contain $C_{2n_1}$ 
  and by the induction hypothesis, the remaining $(n-n_1,1)$-weak ladder contains $C_{2n_2}, \dots, C_{2n_l}$.
  \end{proof}

  \begin{corollary} \label{weak ladder criteria - 2}
  Let $r \in \{1,2\}$, let $2 \leq n_1 \leq \cdots \leq n_l \in \mathbb{Z}^+$, and set $n=\sum_{i=1}^l n_i$. 
  Suppose that $G$ contains a $(n',k)-$weak ladder satisfying $n' \geq n-r$, $k \geq r$, and $n \geq 6k + 12$, and a disjoint ladder $L_{n''}$ for some $n'' \geq n/3$.
  If $G$ does not contain disjoint cycles $C_{2n_1}, C_{2n_2}, \dots, C_{2n_l}$, then $l=2$ and $\lfloor \frac{n+1-r}{2} \rfloor \leq n_1 \leq \frac{n}{2} \leq n_2 \leq \lceil \frac{n+r-1}{2} \rceil. $
  \end{corollary}
  \begin{proof}
  Suppose $G$ contains an $(n',k)-$weak ladder $L$, consisting of two ladders $L_{a_1}$, $L_{a_2}$ with $a_1 + a_2 = n'$ and disjoint paths $P,Q$ such that  $|\mathring{P_1}| + |\mathring{P_2}| = 2k$.
  Let $N = \{n_i : i \in [l]\}$ and let $N_0 = \{n_i \in N: n_i \leq k+r-1 \}$.
  Note that $n' + k \geq n$.
  If there exists $N' \subset N$ such that $k + r \leq \sum_{x \in N'}x \leq \frac{n}{3}$, then $L_{n''}$ contains disjoint $C_{2x}$ for all $x \in N'$; by Lemma \ref{weak ladder criteria - 1}, the $(n',k)$-weak ladder contains the remaining cycles.
  Note that $\sum_{x \in N_0}x \leq \frac{n}{3}$. Indeed, if $\sum_{x \in N_0} x > n/3$, then there exists $N_0' \subset N_0$ such that $k+r \leq (n/3 +1) - (k+r-1) \leq \sum_{x \in N_0'} x \leq n/3$
  
  If $|N \setminus N_0| \geq 3$, then there exists $x \in N \setminus N_0$ such that $k+r \leq x \leq \frac{n}{3}$.
  If $|N \setminus N_0| = 1$, say $N \setminus N_0 = \{y\}$, then we are done as well.
  
  Finally, suppose $N \setminus N_0 = \{y_1,y_2\}$; without loss of generality, suppose $y_1 \leq y_2$. Since $\sum_{x \in N_0}x~\leq~n/3$, we find $C_{2x}$ for all $x \in N_0$ inside $L_{n''}$ . 
  If $y_1 \leq a_2 -1 $, then $C_{2y_1}\subseteq L_{y_1}\subseteq L_{a_2-1}$ and the $(n'-y_1,k)$-weak-ladder obtained by deleting $L_{y_1}$ contains $C_{2y_2}$.
  
  Finally, suppose that $y_1 \geq a_1 + 1 + r$, say $y_1 = a_1 +1 +t$ where $t \geq r$, and let $s := \max \{0 , t- k\}$. We have two cases.
  \begin{itemize}
  \item Assume $s = 0$. 
  Since $y_1 \leq a_1 +1 + k $, an $(a_1 + 1, k)$-weak ladder consisting of $L_{a_1}$ and the first rung of $L_{a_2}$ contains $C_{2y_1}$. 
  Moreover, 
  $$y_2 \leq n - (a_1 + 1 + r) \leq n' + r - (a_1 + 1+ r) = n'-a_1 -1 = a_2 -1.$$
   As such, $L_{a_2-1}$ contains $C_{2y_2}$.
  \item Assume $s > 0$.
  Since $y_1 = a_1 +1 +k +s$, an $(a_1+1+s, k)$-weak ladder consisting of $L_{a_1}$ and the first $s+1$ rung of $L_{a_2}$ contains $C_{2y_1}$. Moreover,
  $$y_2 = n - (a_1 +1 +k+s) \leq n' + r - (a_1 + 1+ r+s) \leq (n' - a_1) - (1+s) \leq a_2 - (1+s).$$
  As such, $L_{a_2 -1 -s }$ contains $C_{2y_2}$.
  \end{itemize}
  Thus $a_2 \leq y_1 \leq a_1 + r$; by a symmetric argument we obtain $a_2 \leq y_1,y_2 \leq a_1+r.$
  If there exists $y \in N_0$, then $L_{n''}$ contains $C_{2y}$ and $y_1 + y_2 \leq n-y \leq n- 2 \leq n' = a_1 + a_2$; then $y_1 = a_1, y_2 = a_2$, which implies that $L_{a_i}$ contains $C_{2y_i}$ for $i \in {1,2}$. 
  Therefore, $a_1 \leq a_2 \leq n_1 \leq n_2 \leq a_1 + r \leq a_1 + 2$, which implies that 
  $$ \left\lfloor \frac{n+1-r}{2} \right\rfloor \leq n_1 \leq \frac{n}{2} \leq n_2 \leq \left\lceil \frac{n+r-1}{2} \right\rceil. $$
  \end{proof}

  We will next show that in special situations it is easy to find a weak ladder. To prove our next lemma we will need the following theorem of Posa \cite{POSA}.
  \begin{theorem} \label{posa Hamilton} \cite{POSA} [L.Posa]
  Let $G$ be a graph on $n \geq 3$ vertices. If for every positive integer $k < \frac{n-1}{2}, |\{v : d_G(v) \leq k \}| < k$ and if, for odd $n, |\{v : d_G(v) \leq \frac{n-1}{2} \} | \leq \frac{n-1}{2}$, then $G$ is Hamiltonian.
  \end{theorem}
  
  First, we will address the case of an almost complete graph.
  \begin{lemma} \label{almost complete}
  Let $\tau \in (0,1/10)$ and $\tau n \geq 100$. Let $G = (V,E)$ be a graph of order $n$ such that there exists $V' \subset V$ such that $|V'| \geq (1-\tau)n$ and for any $w \in V \backslash V', |N(w) \cap V'| \geq 4\tau |V'|$
  where  $V' = \{v \in V: |N(v) \cap V'| \geq (1-\tau)|V'| \}$. Let $u_1,v_1,u_2,v_2 \in V$. The followings hold.
  \begin{enumerate}
  \item There exists $z \in V$ and a ladder $L_{n_1}$ in $G[V \setminus \{u_1,v_1,u_2,v_2, z\}]$ such that $L_{n_1}$ has $\{x_1,y_1\},\{x_2,y_2\}$ as its first, last rung where $x_1 \in N(u_1), y_1 \in N(v_1), z \in N(u_1) \cap N(x_1),x_2 \in N(u_2), y_2 \in N(v_2)$ and
    $n_1 = \lfloor \frac{n - 5}{2} \rfloor $.
  \item Let $x \in N(u_1), y \in N(v_1)$ be such that $x \sim y$. $G$ contains $L_{\lfloor \frac{n-2}{2}\rfloor}$ in $G[V \setminus \{u_1,v_1\}]$ having $\{x,y\}$ as its first rung.
  \item Let $x \in N(u_1)$ and $y \in N(v_1)$ be such that $x \sim y$. For any $z \in N(u_1) \cap N(x)$, $G$ contains $L_{\lfloor \frac{n-3}{2}\rfloor}$ in $G[V \setminus \{u_1 ,v_1,z\}]$ having $\{x,y\}$ as its first rung.
  \item Let $x \in N(u_1) \cap N(v_1)$. $G$ contains $L_{\lfloor \frac{n-1}{2}\rfloor}$ in $G[V \setminus \{u_1\}]$ having $\{x,v_1\}$ as its first rung.
  \item $G$ contains a Hamilton path $P$ having $u_1,v_1$ as its end vertices.
  \end{enumerate}
  We call the vertex $z$ in parts $1$ and $3$ the parity vertex. 
  \end{lemma}
  \begin{proof}
    We will only prove part (1) as the other parts are very similar.
    Let $V'' = V \setminus V'$. Since $|N(u_1) \cap V'|, |N(v_1) \cap V'| \geq 4\tau |V'| > 3\tau n$, there exists $x_1 \in N(u_1) \cap V', y_1 \in N(v_1) \cap V'$ such that $x_1 \sim y_1$ 
    and the same is true for vertices $u_2, v_2$. Let $e_1 = \{x_1,y_1\}, e_2 = \{x_2,y_2\}$. 
    Moreover, since $|N(u_1) \cap N(x_1)| \geq 3 \tau |V'| > 8$, we can choose $z \in N(u_1) \cap N(x_1)$ which is different than any other vertex already chosen.
    
    Now, let $G' = G[V \setminus \{ u_1,v_1,u_2,v_2,x_1,y_1,x_2,y_2, z\}]$ and redefine $V' := V' \cap V(G'), V'' := V'' \cap V(G')$.
    For any $w \in V'', |N(w) \cap V'| \geq 3\tau n - 9  > \tau n \geq |V''|$, so there exists a matching $M_1 \in E(V'', V')$ saturating $V''$.
    Note that $|M_1| \leq |V''| \leq \tau n$. Let $G'' = G[V' \setminus V(M_1)]$. 
    Since $$ \delta(G'') \geq (1-\tau)^2 n - (2\tau n + 9) > (1-5\tau)n > \frac{n}{2} > \frac{|G''|}{2},$$
     $G''$ is Hamiltonian, so there exists a matching of size $\left\lfloor \frac{|G''|}{2} \right\rfloor$ in $G''$, say $M_2$.
    Let $M = M_1 \cup M_2$ and define the auxiliary graph $H = (M,E')$ with the vertex set $M$ and the edge set $E'$ as follows: 
    Let $e' =\{x',y'\}, e'' = \{x'',y''\} \in M$. If $e',e'' \in M_1$ then $\{e',e''\} \notin E'.$
    Otherwise, $\{e',e''\} \in E'$ if $G[e',e'']$ contains a matching of size 2.
    
    If $e \in M_1$ then $d_{H}(e) \geq |N_{H}(e) \cap M_2| > \tau n$, and for any other $e \in M_2$, $d_{H}(e) \geq |N_{H}(e) \cap M_2| \geq |M_2| - \tau |V'| \geq (\frac{1}{2} - 3\tau)n > \frac{|H|}{2}$. 
    Since $|M_1| \leq \tau n$, by Theorem \ref{posa Hamilton}, $H$ contains a Hamiltonian cycle $C$, say $C:= u_1 \dots u_{n'}$ where $n' = \lfloor \frac{n-9}{2} \rfloor$.
    
    Since $d_H(e_1),d_H(e_2) > \frac{|H|}{2}$, there exists $i \in [n']$ such that $u_i \in N(e_1), u_{i+1} \in N(e_2)$
    then $ e_1 u_i C u_{i+1} e_2$, gives a ladder $L_{n_1}$ having $e_1$ and $e_2$ as its first and last rungs with $n_1 = n' + 2 = \lfloor \frac{n - 5}{2} \rfloor$.
  \end{proof}
  
  We call the graph satisfying the condition in Lemma \ref{almost complete} a {\emph{$\tau$-complete graph}}, the vertex set $V'$ the {\emph{major set}} and $V''$ the {\emph{minor set}}.
  \begin{fact} \label{almost complete fact 1}
  If $G$ is $\tau$-complete then for any subset $U$ of the minor set, $G[V \setminus U]$ is still $\tau$-complete.
  \end{fact}  
  Moreover,
  \begin{corollary} \label{attaching two components}
  Let $\tau \in (0,1/10)$. Let $G = (V,E)$ be a graph and $X_1 \subset V, X_2 \subset V$ be two disjoint vertex subsets such that $G[X],G[Y]$ are $\tau$-complete and $|X_1|,|X_2| \geq \frac{100}{\tau}$.
  Suppose that there exist $\{u_1,u_2\}, \{v_1,v_2\} \in E(X_1,X_2)$.
  Then $G[X_1 \cup X_2]$ contains $(n',2)$-weak ladder where $n' \geq \lfloor \frac{|X_1|}{2} \rfloor + \lfloor \frac{|X_2|}{2} \rfloor - 2$.
  Furthermore, if $u_1 \sim v_1$ or $u_2 \sim v_2$ then $G[X_1 \cup X_2]$ contains $(n',1)$-weak ladder where $n' \geq \left\lfloor \frac{|X_1|}{2} \right\rfloor + \left\lfloor \frac{|X_2|}{2} \right\rfloor - 1$.
  \end{corollary}
  \begin{proof}
  For $i \in [2]$, by Lemma \ref{almost complete} (2), $G[X_i]$ contains $L_{\lfloor \frac{|X_i| - 2}{2} \rfloor}$ having $\{x_i,y_i\}$ as its first rung where $x_i \in N(u_i), y_i \in N(v_i) $. 
  By attaching two ladders with $\{u_1,u_2\}, \{v_1,v_2\} $, we obtain a $(n',2)$-weak ladder where $n' \geq \lfloor \frac{|X_1|}{2} \rfloor + \lfloor \frac{|X_2|}{2} \rfloor - 2$ and the "Furthermore" is obvious. 
  \end{proof}
  
  Next, we will address the case of almost complete bipartite graph.
  \begin{lemma} \label{almost complete bipartite}
  Let $\tau \in (0,\frac{1}{100})$. Let $G = (X,Y,E)$ be a bipartite graph with bipartition $X,Y$ such that $n = |X| = |Y|$ and $\tau n \geq 100$.
  Suppose that there exists $X' \subset X, Y' \subset Y$ such that  for any $x \in X', y \in Y'$, $|N(x) \cap Y'| \geq (1-\tau) n, |N(y) \cap X'| \geq (1- \tau) n$ and for any $x \in X \setminus X', y \in Y \setminus Y', |N(x) \cap Y'| \geq 4 \tau n, |N(y) \cap X'| \geq 4 \tau n.$
  Let $e_1,e_2,e_3,e_4$ be  such that for any $e_i$, $i \in [4]$, $|e_i \cap (X' \cup Y')| \geq 1$. 
  Then $G$ contains $L_n$ having $e_i$ as its $f(i)$th rung where $f(1) = 1$ and for any $i \in [3]$, $0 < f(i+1) - f(i) \leq 3$.
  Furthermore, if $|e_i \cap (X' \cup Y')| = 2 $ then we have $|f(i+1) - f(i)| \leq 2$.
  \end{lemma}
  \begin{proof}
  Let $V' = X' \cup Y', V'' = X'' \cup Y''$ and note that $|X'|, |Y'| \geq (1-\tau)n$.
  Let $i \in [3]$. If $|e_i \cap (X' \cup Y')| = 2$ then we can choose $e \in E(X',Y')$ such that $G[e, e_i] \cong K_{2,2}$ and $G[e, e_{i+1}] \cong K_{2,2}$.
  Otherwise, we can choose $e',e'' \in E(X',Y')$ such that $G[e_i,e'], G[e',e''], G[e'', e_{i+1}] \cong K_{2,2}$.
  Hence we obtain a $L_q$ where $q \leq 10$ having $e_i$ as its $f(i)$th rung such that $f : [4] \rightarrow [q]$ satisfies the condition in the lemma.

  Now, let $X' = X' \setminus V(L_q)$ and $Y' = Y' \setminus V(L_q)$, set $X'' = X'' \setminus V(L_q)$ and $Y'' = Y'' \setminus V(L_q)$, and set $X = X' \cup X''$, $Y = Y' \cup Y''$. Note that $V = X \cup Y$.
  For any $x \in X''$, since $|N(x) \cap Y'| \geq 3\tau n > |X''|,$ there exists a matching $M_{X''}$ saturating $X''$.
  Similarly, there exists matching $M_{Y''}$ saturating $Y''$. 
  
  Let $M_1 = M_{X''} \cup M_{Y''} $ and $G' = G[V \setminus V(M_1)]$.
  For each $e = \{x_i,y_i\} \in M_1$, we can pick $x_i',x_i'' \in (N(y_i) \cap X') $, $y_i',y_i'' \in N(x_i) \cap Y' $, so that all vertices are distinct and $x_i' \sim y_i', x_i'' \sim y_i''$.
  This is possible because $|N(y_i) \cap X'|  > 3\tau n \geq 3|M_{X''}| $ and $|N(x_i) \cap Y'| > 3\tau n \geq 3|M_{Y''}|. $
  Then $G[\{x_i,y_i,x_i',x_i'',y_i',y_i''\}]$ contains a 3-ladder, which we will denote by $L_i$.  
  We have $|X''|+|Y''| = m$ 3-ladders each containing exactly one vertex from $X'' \cup Y''$.

  Let $X''' = X' \setminus (\cup_{i \in [m]}L_i)$,  $Y''' = Y' \setminus (\cup_{i \in [m]}L_i)$.
  Then $|Y'''|= |X'''| \geq (1-3\tau)n - q > n/2.$
  For any $x \in X'''$, 
  \begin{align*}
  |N(x) \cap Y'''| &\geq |Y'''| - \tau n - |(V(\bar{M_1}) \cap Y''' |  \\                   
                   &> |Y'''| - 4\tau n \\
                   & > (1-8 \tau)|Y'''| > \frac{|Y'''|}{2},
  \end{align*}
  so there exists a matching $M_2$ saturating $X'''$. Define the auxiliary graph $H$ as follows. For every $L_i$, consider vertex $v_{L_i}$ and let 
  $$ V(H) = \{v_{L_i} : i \in [m] \} \cup \{e : e \in M_2 \}. $$ 
  For $e = \{a_i,b_i\} , e' = \{a_j,b_j\} \in M_2$, we put $\{e,e'\} \in E(H)$ if $G[\{a_i,a_j\} ,\{b_i,b_j\}] = K_{2,2}$; for $v_{L_i} \in V(H)$ and $e = \{a_j,b_j\} \in M_2$, we put $\{v_{L_i}, e\} \in E(H)$ if $a_j \in N(y_i') \cap N(y_i'')$ and $b_j \in N(x_i') \cap N(x_i'')$.
  Then $\delta(H) \geq |H| - 10 \tau n > |H|/2$ and so $H$ is Hamiltonian;
  this gives the desired ladder $L_n$ by attaching $L_q$ as its first $q$ rungs.
  \end{proof}
  
  Similarly, we also have another lemma for the case that $G$ is almost complete bipartite, but the sizes of the sets in the bipartition differ.
  \begin{lemma} \label{almost complete bipartite - 1}
  Let $\tau \in (0,\frac{1}{100})$ and $C \in \mathbb{R}$ be such that $\tau C \leq \frac{1}{300}$. Let $G = (X,Y,E)$ be a bipartite graph with bipartition $X,Y$ such that $n = |Y| \leq |X| \leq C n$ and $\tau n \geq 100$.
  Suppose that there exists $X' \subset X, Y' \subset Y$ such that for any $y \in Y'$ we have $|N(y) \cap X| \geq (1-\tau) |X|$, and for any $x \in X$ we have $|N(x) \cap Y'| \geq (1- \tau) |Y'|.$ Suppose also every $y \in Y \setminus Y'$ satisfies $|N(y) \cap X| \geq 4 \tau |X|.$
  Let $e_1,e_2,e_3,e_4$ be  such that for any $e_i$, $i \in [4]$, $|e_i \cap (X' \cup Y')| \geq 1$. 
  Then $G$ contains $L_n$ having $e_i$ as its $f(i)$th rung where $f(1) = 1$ and for any $i \in [3]$, $0 < f(i+1) - f(i) \leq 3$.
  Furthermore, if $|e_i \cap (X' \cup Y')| = 2 $ then we have $|f(i+1) - f(i)| \leq 2$.
  \end{lemma}
  \begin{proof}
  The proof is similar to the proof of Lemma \ref{almost complete bipartite}. In the same way, we obtain $L_q$ containing $e_1,e_2,e_3,e_4$ in desired positions and let $X = X \setminus V(L_q), Y' = Y' \setminus V(L_q),  Y'' = (Y \setminus Y')
   \setminus V(L_q)$, so that $V = X \cup Y' \cup Y''$.
  For any $y \in Y''$, since $|N(y) \cap X| \geq 3 \tau |X| > |Y''|,$ there exists a matching $M$ saturating $Y''$.

  Let $G' = G[V \setminus V(M)]$.
  For each $e = \{x_i,y_i\} \in M$, we can pick $x_i',x_i'' \in N(y_i) $, $y_i',y_i'' \in N(x_i) \cap Y' $, so that all vertices are distinct and $x_i' \sim y_i', x_i'' \sim y_i''$.
  This is possible because $|N(y_i) \cap X'|  > 3 \tau |X| \geq 3|M| $ and for any $x_i,x_i',x_i''$, $|N_{G'}(x_i) \cap Y'|,|N_{G'}(x_i') \cap Y'|,|N_{G'}(x_i'') \cap Y'| > (1-\tau)n > (\frac{1}{2} + 3\tau)n$.
  Then $G[\{x_i,y_i,x_i',x_i'',y_i',y_i''\}]$ contains a 3-ladder, which we will denote by $L_i$.  
  We have $|Y''| = m$ 3-ladders each containing exactly one vertex from $Y''$.

  Let $Y''' = Y' \setminus (\cup_{i \in [m]}L_i)$ and choose $X''' \subset X \setminus (\cup_{i \in [m]}L_i)$ such that $|X'''| = |Y'''|$.
  Then $|Y'''|= |X'''| \geq (1-3\tau)n - q > n/2.$
  For any $x \in X'''$, 
  \begin{align*}
  |N(x) \cap Y'''| &\geq |Y'''| - \tau n - |(V(\bar{M_1}) \cap Y''' |  \\                   
                   &> |Y'''| - 4\tau n \\
                   & > (1-8 \tau)|Y'''| > \frac{|Y'''|}{2},
  \end{align*}
  so there exists a matching $M_2$ saturating $X'''$. Define the auxiliary graph $H$ as follows. For every $L_i$, consider vertex $v_{L_i}$ and let 
  $$ V(H) = \{v_{L_i} : i \in [m] \} \cup \{e : e \in M_2 \}. $$ 
  For $e = \{a_i,b_i\} , e' = \{a_j,b_j\} \in M_2$, $\{e,e'\} \in E(H)$ if $G[\{a_i,a_j\} ,\{b_i,b_j\}] = K_{2,2}$ and for $v_{L_i} \in V(H), e = \{a_j,b_j\} \in M_2$, $\{v_{L_i}, e\} \in E(H)$ if $a_j \in N(y_i') \cap N(y_i'')$, $b_j \in N(x_i') \cap N(x_i'')$.
  Then $\delta(H) \geq |H| - 20 C\tau n > |H|/2$ and then $H$ is Hamiltonian,
  which gives a desired ladder $L_n$ by attaching $L_q$ as its first $q$ rungs.
  \end{proof}
  
  A {\emph{T-graph}} is graph obtained from two disjoint paths $P_1= v_1, \dots, v_m$ and $P_2= w_1, \dots, w_l$ by adding an edge $w_1v_i$ for some $i=1, \dots, m$. In \cite{CK}, it is shown that if $P=V_1, \dots, V_{2s}$ is a path consisting of pairwise-disjoint sets $V_i$ such that
  $|V_1|=l-1$, $|V_{2s-1}|=l+1$, $|V_i|=l$ for every other $i$, and in which $(V_i,V_{i+1})$ is $(\epsilon, \delta)$-super regular for suitably chosen $\epsilon$ and $\delta$, then
  $G\left[\bigcup V_i\right]$ contains a spanning ladder.
  We will use this result in one part of our argument but in many other places the following, much weaker statement will suffice.
  \begin{lemma}\label{ladder}
    There exist $0< \epsilon$, $10\sqrt{\epsilon} < d < 1$, and $l_0$ such that the following holds.
    Let $P=V_1, \dots, V_{r}$ be a path consisting of pairwise-disjoint sets $V_i$ such that
    $|V_i|=l\geq l_0$ and in which $(V_i,V_{i+1})$ is $(\epsilon, d)$-super regular. In addition, let $x_1\in V_1, x_2\in V_2$. Then $G\left[\bigcup V_i\setminus \{x_1, x_2\}\right]$ contains a ladder $L$ such that the first rung of $L$ is in $N(x_1)\cap V_2$, $N(x_2)\cap V_1$ and $|L| \geq (1- 5\sqrt{\epsilon}/d)rl$.
  \end{lemma}
  \begin{proof}
  We will construct $L$ in a step by step fashion. Initially, let $L:=\emptyset$ and let $k \in [2]$. We have $|N(x_k)\cap V_{3-k}| \geq dl > \epsilon l$ and so there exist $x_1', x_2'$ such that $x_k'\in N(x_k)$, $x_1'x_2'\in E$ and $|N(x_k')\cap V_{k}\setminus L|\geq dl - 1 \geq 2\sqrt{\epsilon}l$. For the general step, suppose $x_1\in V_1,x_2\in V_2$ are the endpoints of $L$ and $|N(x_k)\cap V_{3-k}\setminus L| \geq 2\sqrt{\epsilon} l $.
  Let $U_k:=V_{k} \setminus L$ and suppose $|U_k| \geq 5\sqrt{\epsilon} l/d$. Then, by Lemma \ref{slicing lemma}, $(U_k,N(x_k)\cap V_{3-k}\setminus L)$ is $\sqrt{\epsilon}$-regular with density at least $d/2$.  Thus all but at most $\sqrt{\epsilon}l$ vertices $v\in N(x_k)\cap V_{3-k}\setminus L$ have $|N(v) \cap U_k| \geq (\frac{d}{2}-\sqrt{\epsilon})|U_k| \geq 2\sqrt{\epsilon}l+1$. 
  Since $|N(x_k)\cap V_{3-k}\setminus L| \geq 2\sqrt{\epsilon} l$, there are $A_k \subset N(x_k)\cap V_{3-k}\setminus L$ such that $|A_k| \geq \sqrt{\epsilon}l$ and every vertex $v \in A_k$ has $|N(v) \cap U_k| \geq 2\sqrt{\epsilon}l+1$. Hence there exist $x'_1 \in A_1, x'_2 \in A_2$ such that $x_1'x_2' \in E$ and $|N(x'_k)\cap V_{k}\setminus (L \cup \{x_k\})| \geq 2\sqrt{\epsilon} l$ and we can add one more rung to $L$ from $V_1\times V_2$. To move from $(V_1,V_2)$ to $(V_3,V_4)$ suppose $L$ ends in $x_1\in V_1, x_2\in V_2$. 
  Pick $x_1'\in N(x_1) \cap V_{2}\setminus L$ so that $|N(x_2)\cap N(x_1')\cap V_3|\geq 2\sqrt{\epsilon}l$. Note that $|N(x_2)\cap V_3| \geq dl, |N(x_1) \cap V_2 \setminus L| \geq 2\sqrt{\epsilon}l$ and so $x_1'$ can be found in the same way as above. Next find $x_2', x_3 \in N(x_2)\cap N(x_1')\cap V_3$ such that $|N(x_2') \cap N(x_3) \cap V_4| > 0$, and finally let $x_4 \in N(x_2')\cap N(x_3)\cap V_4$. Then $\{x_3, x_4\} \in E, x_3 \in N(x_2) \cap N(x'_1) \cap V_3, x_4 \in N(x'_2) \cap V_4 $ and $|N(x_3) \cap (V_4 \setminus \{x_4\})|, |N(x_4) \cap (V_3 \setminus \{x'_2,x_3\})| \geq dl - 2 \geq 2\sqrt{\epsilon} l $.
  \end{proof}
  We will need the following observation.
  \begin{fact}\label{short-paths}
  Let $G$ be a 2-connected graph on $n$ vertices such that $\delta(G)\geq \alpha n, n > \frac{10}{\alpha^2}$ and let $U_1, U_2$ be two disjoin sets such that $|U_i|\geq 2$. 
  Then there exist two disjoint $U_1-U_2$ paths $P_1,P_2$ such that $|P_1|+|P_2| \leq \frac{10}{\alpha}$.
  \end{fact}
  \begin{proof}
  Let $P_1, P_2$ be two $U_1-U_2$ paths such that $|P_1|+|P_2|$ is the smallest. Without loss of generality, $|P_1| \leq |P_2|$.
  Note that both paths are induced subgraphs and suppose $P_2 := v_1 \dots v_l$, $l > 5/\alpha$.
  Let $A = \{v_{3i} : i \in [ \frac{l}{3}] \}$. If for any $x,y \in A$, $|N(x) \cap N(y) | \leq 1$ then 
  $$|\cup_{v \in A} N(v) \cap (V \setminus V(P_2)) | \geq \sum_{i = 0}^{|A|} \max \{ (\alpha n - 2 - i) , 0 \} > n,$$
  a contradiction.
  Hence there exist two vertices $x,y$ in $P_2$ such that $dist_{P_2}(x,y)>2$ and $|N_G(x)\cap N_G(y)|\geq 2$. 
  Then $N_G(x)\cap N_G(y) \cap (V\setminus V(P_1))=\emptyset$ or we get a shorter $U_1-U_2$. Thus $|N_G(x) \cap N_G(y) \cap V(P_1)| \geq 2$ and we again get shorter disjoint $U_1-U_2$ paths.
  \end{proof}
  As our last fact in this section we will show that a component in our graph either contain two disjoint paths of total length much bigger than its minimum degree or the component has a very specific structure.
  \begin{theorem} \label{no second path}
  Let $C$ be a component in a graph $G$ which satisfies $|C| \geq 2\delta(G)$. If $G[C]$ does not contain a Hamiltonian path then 
  either there exists a path $P_1$ such that for any $v \in V(C) \backslash V(P_1)$, $N(v) \subset V(P_1)$ or there exist two disjoint paths $P_1,P_2$ such that $|V(P_1)|+|V(P_2)| > 3\delta(G)$.
  \end{theorem}
  \begin{proof}
  Let $P_1$ be a maximum path in $C$, say $P_1 = v_1, \ldots ,v_r$. If $P_1$ is a Hamiltonian path or $G[V(C) \backslash V(P_1)]$ is independent then we are done; thus we may assume that there exists a path in $G[V(C) \backslash V(P_1)]$, say $P_2 = u_1, \dots, u_s$ such that $s \geq 2$.
  Let
  \begin{align*}
  A = \{i : v_i \in N(v_1) \cap V(P_1)\}&, A^{-} = \{i-1 : i \in A\},\\
  B = \{i : v_i \in N(v_r) \cap V(P_1)\}&, B^{+} = \{i+1 : i \in B\} .
  \end{align*}
  If $G[V(P_1)]$ contains a cycle of length at least $|V(P_1)| - 1$ then it gives a longer path by attaching $P_2$ to the cycle. Therefore,
  $$A^{-} \cap B^{+} = \emptyset,$$
  which implies that
  $$|A^{-} \cup B^{+}| \geq 2 \delta(G).$$
  By the maximality of $P_2$,
  \begin{align*}
  N(u_1) \subset V(P_2) \cup V(P_1).
  \end{align*}
  By the maximality of $P_1$,
  $$ N(u_1) \cap (A^{-} \cup B^{+}) = \emptyset. $$
  Therefore, 
  $$\delta(G) \leq d(u_1) \leq r- 2\delta(G) + s-1 ,$$
  which implies that 
  $$ |V(P_1)| + |V(P_2)| = r + s \geq 3\delta(G) + 1.$$
  \end{proof}

  \section{The first non-extremal case}\label{sec:nonextr1}
  In this section we will address the case when $G$ is non-extremal and  $\alpha n\leq \delta(G) \leq (1/2-\gamma) n$ for some $\alpha, \gamma >0$. In order to do this, of course, we must first define what it means to be extremal.
  \begin{definition} \label{def extremal}
  Let $G$ be a graph with $\delta(G) = \delta$.
  We say that $G$ is {\emph{$\beta$-extremal}} if there exists a set $B \subset V(G)$ such that
   $|B| \geq (1-\delta/n-\beta)n$ and all but at most $4\beta n$ vertices $v\in B$ have $|N(v) \cap B| \leq \beta n$.
  \end{definition}
  Then the main theorem in this section follows.
  \begin{theorem} \label{non extremal main}
  Let $\alpha, \gamma \in (0,\frac{1}{2})$ and let $\beta>0$ be such that $\beta < (\frac{\alpha}{400})^2 \leq \frac{1}{640000}$. Then there exists $N(\alpha,\gamma)\in \mathbb{N}$ such that for all $n \geq N$ the following holds.
  For every 2-connected graph $G$ on $n$ vertices with $\alpha n\leq \delta(G) \leq (1/2-\gamma) n$ which is not $\beta$-extremal and every $n_1, \dots, n_l\geq 2$ such that $\sum n_i = \delta$
  \begin{itemize}
  \item[(i)] $G$ contains disjoint cycles $C_{2n_1}, C_{2n_2}, \dots, C_{2n_l}$ or
  \item[(ii)] $\delta$ is even, $n_1=n_2=\frac{\delta}{2}$ and $G$ is one a graph from Example \ref{ex1}.
  \end{itemize}
  \end{theorem}
  \begin{proof}
  Fix constants $d_1 := \min \left\{\frac{\alpha^6}{10^{10}}, \frac{\gamma}{10} , \beta^2 \right\} , d_2:= \frac{d_1}{2}$ and let $\epsilon_1,\epsilon_2,\epsilon_3$ be such that $\epsilon_1 < 300 \epsilon_1 <\epsilon_2 <  \epsilon_2^{1/4} < \frac{\epsilon_3}{10} < 10 \epsilon_3 < d_2$. 
  Applying Lemma \ref{regularity} with parameters $\epsilon_1$ and $m$, we obtain our necessary $N = N(\epsilon_1,m)$ and $M = M(\epsilon_1,m)$. Let $N(\alpha) = \max \left\{N, \left\lceil \frac{100M}{\alpha \epsilon_3} \right\rceil \right\}$ and let $G$ be an arbitrary graph with $|G| = n \geq N(\alpha)$ and $\delta = \delta(G) \geq \alpha n$.
  By Lemma \ref{regularity} and some standard computations, we obtain an $\epsilon_1$-regular partition $\{V_0, V_1, ..., V_t \}$ of $G$ with $t\in[m,M]$, $|V_0|\le\epsilon_1 n$ and such that there are at most $\epsilon_1 t$ pairs of indexes $\{i,j\}\in \binom{[t]}{2}$ such that $(V_i, V_j)$ is not $\epsilon_1$-regular.
  
  Let $l := |V_i|$ for $i \geq 1$ and note that
  \begin{align*}
  (1-\epsilon_1)\frac{n}{t} \leq l \leq \frac{n}{t}.
  \end{align*}
  
  Now, let $R$ be the cluster graph with threshold $d_1$, that is, given $\{V_0, V_1, ..., V_t \}$ as above,
  $V(R) = \{V_1, \dots, V_t\}$ and $E(R) = \{ V_iV_j : (V_i,V_j) \text{ is } \epsilon_1 \text{-regular with } d(V_i,V_j) \geq d_1  \}$. In view of the definition of $\epsilon_1$ and $d_1$ we have the following,
  \begin{equation}
  \label{minimum degree of cluster graph}
  \delta(R) \geq (\delta/n - 2d_1)t
  \end{equation}
  
  \begin{lemma} \label{long path}
  Let $C$ be a component in $R$ which contains a T-graph $H$ with $|H| \geq (2\delta/n + \epsilon_3)t$. 
  Then $G$ contains a $(n',r)-$weak ladder where $n' \geq \delta + r$.
  \end{lemma}
  \begin{proof}
  Since $\Delta(H) \leq 3$, by Lemma \ref{regular-graph} applied to $H$ there exist  subsets $V_i'\subseteq V_i$ for every $V_i \in V(H)$ such that $(V_i', V_j')$ is $(\epsilon_2, d_2)$-super-regular for every $V_iV_j\in H$ and $$ \left| V_i' \right| \geq (1-\epsilon_2)l. $$
  
  Let $P=U'_1, \dots, U'_s$, $Q= U'_i,W'_1, \dots, W'_r$ denote the two paths forming $H$. Note that if $i + r \geq (2\delta/n +\epsilon_3)$, then $G[\bigcup_{j=1}^{i} U_j' \cup \bigcup_{j=1}^{r}W_j' ]$ contains a ladder on $m$ vertices where
   $$ m \geq (2\delta/n +\epsilon_3)(1-\epsilon_2)(1-\epsilon_1)n \geq 2\delta.$$
  Otherwise, let $x \in U'_{i+1}, y\in U'_{i+2}$. There is an $x,z$-path $P$ on $r+1$ vertices for some $z\in W'_{r-1}$ and a $y,w$-path $Q$ on $r+1$ vertices for some $w\in W'_{r-2}$ which is disjoint from $P$. By Lemma \ref{ladder},
  there is a ladder $L'$ on  $(i+r)(1-\epsilon_2)(1-5\sqrt{\epsilon_2}/d_2)l$ vertices in $G[U'_1\cup \cdots U'_i \cup W'_1\cup \dots W'_r]$ which ends at $z' \in N(z) \cap W'_{r}$ and $w' \in N(w)\cap W'_{r-1}$ and a ladder $L''$ on $(s-i) (1-\epsilon_2)(1-5\sqrt{\epsilon_2}/d_2)l$ vertices in  $G[U'_{i+2}\cup U'_s]$ which ends at $x' \in N(x)\cap U'_{i+2}$ and $y' \in N(y)\cap U'_{i+3}$
  such that $L \cap (P \cup Q), L' \cap (P \cup Q) = \emptyset$.
  Then $|L'| + |L''| \geq  (2\delta/n + \epsilon_3)t(1-\epsilon_2)(1-5\sqrt{\epsilon_2}/d_2)l \geq 2\delta + \frac{\epsilon_3 n}{2}, |P|=|Q|=r+1$ and $\frac{\epsilon_3 n}{4} - (r+1) \geq 0 $. 
  Thus $L_1 \cup P' \cup Q' \cup L_2$ contains a $(n',r+1)-$weak ladder where $P' = x'Pz', Q' = y'Qw'$ and $n' - (r+1) \geq \delta.$
  \end{proof}
  
  \begin{lemma} \label{connected R}
  Let $C$ be a component in $R$ and suppose $|C|\geq (2 \delta/n +\epsilon_3)t$. Then either there is a T-graph $H$ such that $|H| \geq (2\delta/n + \epsilon_3)t$, or there is a set $\mathcal{I} \subset V(C)$ such that $|\mathcal{I}| \geq |C| - (\delta/n +8d_1)t$ and $||R[\mathcal{I}]|| = 0$.
  \end{lemma}
  \begin{proof}
  Let $P_1=V_1, \dots, V_s$ be a path of maximum length in $C$ and subject to this  is such that
  $||R[V(C) \setminus V(P_1)]||$ is maximum. By Theorem \ref{no second path}  we may assume that
  $s < (2\delta/n + \epsilon_3)t$ and that for any $W \in V(C) \backslash V(P_1)$, $N(W) \subset V(P_1)$
  (i.e. $||R[V(C) \setminus V(P_1)]||=0$). Let $W \in V(C) \backslash V(P_1)$ be arbitrary and let
  \begin{align*}
  \mathcal{W}&= \{i \in [s] : V_i \in N(W) \},\\
   \mathcal{W}^{+} &= \{i \in [s] : i-1 \in \mathcal{W}, i+1 \in \mathcal{W} \},\\
   \mathcal{W}^{++} &= \{i \in [s] : i \in \mathcal{W}, i-1,i-2 \notin \mathcal{W} \}.
  \end{align*}
  Since $P_1$ is a longest path, $\mathcal{W}\cap \mathcal{W}^{+}=\emptyset$.
  
  In addition, note that $|\mathcal{W}^{+}| + |\mathcal{W}^{++}| + 1 = |N_R(W)|$. As a result, if $|\mathcal{W}^{+}| = (\delta/n - C d_1)t, C \geq 7$ then $|\mathcal{W}^{++}| \geq (C - 2)d_1 t$. But then,
  \begin{align*}
  |V(P_1)| &\geq 2|\mathcal{W}^{+}| + 3 |\mathcal{W}^{++}| \\ 
       &\geq 2(\delta/n - C d_1)t + 3\cdot (C - 2) d_1 t \\ 
       &\geq (2\delta/n + (C - 6)d_1) t > 2(\delta/n + \epsilon_3) t > |V(P_1)|.
  \end{align*}
  Thus we may assume that $|\mathcal{W}^{+}| > (\delta/n - 7 d_1)t$.
  Let $\mathcal{I} := \{V_i| i \in \mathcal{W}^{+} \}\cup (V(R) \backslash V(P_1)) .$ Then $|\mathcal{I}| \geq |C|- (|V(P_1)|-|\mathcal{W^+}|) \geq |C|- (\delta/n + 8 d_1)t$. We will  show that $\mathcal{I}$ is an independent set in $R$.
  Clearly $V(C) \setminus V(P_1)$ is independent.  Suppose there is $W' \in V(C) \backslash V(P_1)$ such that  for some $i \in \mathcal{W^+}$, $V_i\in N_R(W')$.  Let $P_1'$ be obtained from $P_1$ by  exchanging $V_i$ with $W$ and note that the length of $P_1'$ is equal to the length of $P_1$ but $||R[V(C) \setminus V(P_1')]||\neq 0$ contradicting the choice of $P_1$.
  Now suppose $V_iV_j \in R$ for some $i,j\in \mathcal{W}^+$, with $i<j$. Then $P_1':=V_sP_1V_{j+1}WV_{i+1}P_1V_jV_iP_1V_1$ is a longer path.
  \end{proof}
  
  In the following lemma, we show that for graphs whose reduced graphs are connected, either the graph contains a $\delta$-weak ladder, hence it includes the claimed number of cycle lengths, or again it is very nearly our extremal structure.
  
  \begin{lemma}\label{done-or-extremal-claim}
  If $R$ is connected, then either $G$ contains a $(n', r)-$weak ladder where $n' \geq \delta + r$, or there exists a set $V' \subset V$ such that $|V'| \geq (1-\delta/n - \beta) n$, such that all but at most $4 \beta n$ vertices $v \in V'$ have
  $|N_{G'}(v)| \le \beta n$ where $G' = G[V']$.
  \end{lemma}
  \begin{proof}
  Since $2\delta/n + \epsilon_3 \leq 2(1/2 - \gamma) + \epsilon_3 \leq 1$, By Claim \ref{connected R} and Claim \ref{long path}, we may assume that there is
  $I \subset V(R)$ such that $|I| \geq |C| - (\delta/n +8d_1)t= (1- \delta/n -8d_1)t$ and $||R[I]|| = 0$.
  Let $V' = \cup_{X \in I} X$. Then
  \begin{align*}
  |V'| = l|I| \geq l (1 - \delta/n - 8d_1)t  \geq (1 - \delta/n - 9d_1)n\geq (1-\delta/n -\beta)n.
  \end{align*}
  Let $W = \{w \in V' : |N_{V'}(w)| \geq \sqrt{d_1}n \}$. We claim that $|W| < 4\sqrt{d_1}n \leq 4\beta n$. Suppose  otherwise. Then we have
  \begin{align*}
  ||G[V']|| \geq \frac{4\sqrt{d_1}n \cdot \sqrt{d_1}n}{2} = 2 d_1 n^2.
  \end{align*}
  which implies that there is at least one edge in $R[I]$. Indeed, there are at most $\epsilon_1 t^2 l^2 \leq \epsilon_1 n^2$ edges in irregular pairs, at most $d_1 t^2 l^2\leq d_1 n^2$ edges in pairs $(A,B)$ with $d(A,B)\leq d_1$, and at most $t \binom{l}{2} < \epsilon_1 n^2$ edges in $\bigcup_{i\geq 1} G[V_i]$.
  \end{proof}
  
  Thus from Lemma \ref{done-or-extremal-claim} we are either done or there is a set $V' \subset V$ such that $|V'| \geq (1-\alpha - \beta)n$, such that all but at most $4 \beta n$ vertices $v\in V'$ have
  $|N_{G'}(v)|\le \beta n$. The latter case will be addressed in the section which contains the extremal case.
  
  However, we are not done yet with the non-extremal case because $R$ can be disconnected. Indeed, it is this part of the argument which requires careful analysis and uses the fact that $G$ is 2-connected. We will split the proof into lemmas based on the nature of components in $R$ and will assume
  in the rest of the section  that $R$ is disconnected.
  
  \begin{lemma}\label{no bipartite}
  If $R$ is disconnected and contains a component $C$ which is not bipartite and a component $C'$ such that $|C'| > (\delta/n +3d_1)t$ then $G$ contains a $(n',r)-$weak ladder for some $n' \geq \delta + r$.
  \end{lemma}
  \begin{proof} Note that $C$ and $C'$ can be the same component. Let $C,C'$ be two components such that $|C|+|C'|\geq (2\delta/n +d_1)t$ and suppose $C$ in not bipartite path. 
  Then there exist path $P=V_1,\dots, V_s$ in $C$  and $Q=U_1, \dots, U_r$ in $C'$  such that $|P|+|Q| \geq (2\delta/n + d_1)t$. In addition, $C$ contains an odd cycle $B$.
  
  Let $\bar{P}$ be obtained from $P$ be applying Lemma \ref{regular-graph} and let $\bar{Q}$ be obtained from $Q$ by applying Lemma \ref{regular-graph} and let $V_1',\dots, V_s', U_1', \dots, U_r'$ denote the modified clusters.
  Let $U_1:= \bigcup_{V\in \bar{P}} V$, $U_2:= \bigcup_{V \in \bar{Q}} V$. Since $G$ is 2-connected, from Fact \ref{short-paths}, there exist two disjoint $U_1-U_2$ paths $Q_1,Q_2$ in $G$ such that $|Q_1|+|Q_2| \leq \frac{10}{\alpha}$. Let $\{x_k,y_k\} = (V(Q_1) \cup V(Q_2) \cap U_k$.
  We will extend $Q_1, Q_2$ to paths $Q_1', Q_2'$, so that $Q_1'\cap Q_2' =\emptyset$, the endpoints of $Q_1'$ are in $U_1'$, $V_1'$,
  the endpoints of $Q_2'$ are in $U_2'$, $V_2'$ and $|Q_1'|=|Q_2'|\leq K$ for some constant $K$ which depends on $\alpha$ only.
  For $C'$ we simply find short paths from $x_2,y_2$ to $U_1'$, $U_2'$, that is, let $x_2'\in U_1'$, $y_2'\in U_2'$ and find paths $S_1,S_2$ so that $S_1\cap S_2=\emptyset$, $S_1$ is an $x_2',x_2$-path, $S_2$ is a $y_2',y_2$-path, $|S_i| \leq r$ and $||S_1| -|S_2||\leq 1$.
  Let $S_i':= S_i\cup Q_i$. Note that $|S_i'| \leq r + \frac{10}{\alpha}$ but the paths can have different lengths. Let $R_1$ be a path in $G[C]$ on at most $|C|$ vertices from $x_1$ to a vertex $x_1' \in V_1'$ which does not intersect $S_1'$. Note that for every $V\in C$, $|V\cap (S_1'\cup S_2'\cup R_1)|$ is a constant and so if $(V,W)$ is $(\epsilon, d)$-super-regular then $(V \setminus (S_1'\cup S_2'\cup R_1),  W \setminus (S_1'\cup S_2'\cup R_1) )$ is $(2\epsilon, d/2)$-super-regular.
  Consequently, using the fact that $C$ contains an odd cycle, it is possible to find a path $R_2$ from $y_2$ to a vertex $y_2'\in V_2$ so that $|R_2|\leq |C|$, $R_2 \cap (S_1'\cup S_2'\cup R_1) =\emptyset$, and $|R_1|+|S_1'|,|R_2|+|S_2'|$ have the same parity. If $|R_1|+|S_1'|>|R_2|+|S_2'|$, then use $(V_1',V_2')$ to extend $R_2$ so that the equality holds. Let $Q_1', Q_2'$ be the resulting paths. Note that
  $|Q_1'|+|Q_2'|$ is constant and since $|P|+|Q| \geq (2\delta/n + d_1)t$ we can find two ladders $L_{n_1}$ in $G[P]$, $L_{n_2}$ in $G[Q]$ such that $n_1+n_2 \geq \delta + d_1 n/4$, $(L_1\cup L_2) \cap (Q_1'\cup Q_2') =\emptyset$ and such that $L_i$ ends in $N(x_i'),N(y_i')$.
  \end{proof}
  Next we will address the case when all components are bipartite.
  
  \begin{lemma} \label{all bipartite}
  If $R$ is disconnected and every component is bipartite, then $G$ contains either $L_{\delta}$ or a $(n',r)-$weak ladder for some $n',r$ such that $n' \geq \delta + r$.
  \end{lemma}
  \begin{proof}
  Let $\xi:= 20d_1/\alpha^2$, $\tau := 20\sqrt{d_1}/ \alpha^2$ and let $q$ be the number of components in $R$ and let $D$ be a component in $R$. Then $D$ is bipartite and so $|D|\geq 2\delta(R)\geq 2(\delta/n-2d_1)t$. 
  Thus, in particular, $q \leq 1/(2(\delta/n-2d_1))\leq n/\delta$.
  
  For a component $D$ in $R$, if $|D| \geq (2\delta/n +\epsilon_3)t$, then by Lemma \ref{connected R} and Lemma \ref{long path}, we may assume that there is an independent set $I \subset V(D)$ such that $|I| \geq |D| - (\delta/n + 8d_1)t$.
  Suppose components are $D_1, D_2, \dots, D_q$ and $D_i$ has bipartition $A_i,B_i$ such that $|A_i| \leq |B_i|$. Then, we have
  $$ (\delta/n -2d_1)t \leq |A_i| \leq (\delta/n +8d_1)t$$
  and $|B_i|\geq (\delta/n - 2d_1)t$.
  Let $X_i:= \bigcup_{W\in A_i}W, Y_i:= \bigcup_{W\in B_i}W$ and $G_i:=G[X_i, Y_i]$. Then
  $$ \delta -3d_1n \leq |X_i| \leq \delta + 8d_1n$$
  and $$\delta -3d_1n \leq |Y_i|.$$
  In addition, since $B_i$ is independent in $R$,
  $||G_i|| = e(X_i,Y_i) \geq \delta |Y_i| - 2d_1 n^2\geq |X_i||Y_i| - 10d_1n^2 \geq (1-\xi)|X_i||Y_i|.$
  
  Let $X_i':=\{x \in X_i| |N_G(x)\cap Y_i| \geq (1-\sqrt{\xi})|Y_i|\}$ and note that $|X_i'| \geq (1- \sqrt{\xi})|X_i| \geq \frac{2 \delta}{3}$. Similarly let $Y_i':=\{y\in Y_i| |N_G(y) \cap X_i| \geq (1-\sqrt{\xi})|X_i|\}$ and note that $|Y_i'| \geq (1-\sqrt{\xi})|Y_i| \geq \frac{2 \delta}{3}$. Let $G':=G[X_i',Y_i']$ and note that
  for every vertex $x \in X_i'$,
  \begin{equation}\label{eq-deg}
  |N_G(x) \cap Y_i'| \geq (1- 2\sqrt{\xi})|Y_i'| ,
  \end{equation}
   and the corresponding statement is true for vertices in $Y_i'$. 

  Let $V_0':= V_0 \cup \bigcup_i((X_i\setminus X_i') \cup (Y_i\setminus Y_i'))$ and note that $|V_0'| \leq (2\sqrt{\xi} +\epsilon_1)n \leq 3\sqrt{\xi}n.$
  Then for every vertex $v \in V_0'$ we have $|N_G(v) \cap (V(G)\setminus V_0')| \geq \delta/2$. Thus, since the number of components is at most $n/\delta$, for every $v\in V_0'$ there is $i\in [q]$ such that $|N_G(v)\cap X_i'| +|N_G(v) \cap Y_i'| \geq \delta^2/(2n)$ and we assign $v$ to $Y_i'$ ($X_i'$) if $|N_G(v)\cap X_i'| \geq \delta^2/(4n)$ $(|N_G(v) \cap Y_i'| \geq \delta^2/(4n))$ so that every $v$ is assigned to exactly one set.
  Let $X_i''$, ($Y_i''$) denote the set of vertices assigned to $X_i'$ ($Y_i'$) and let $V_i':= X_i'\cup X_i''\cup Y_i'\cup Y_i''$.
  
  First assume that there exists $i$ such that $\min\{|X_i'\cup X_i''|, |Y_i'\cup Y_i''|\} \geq \delta$.
  If $|X_i'|, |Y_i'| \leq \delta$ then by removing some vertices from $X_i'' \cup Y_i''$, we get $|X_i' \cup X_i''| = |Y_i' \cup Y_i''| = \delta$ 
  and by Lemma \ref{almost complete bipartite}, we obtain $L_{\delta}$.
  If $|X_i'| > \delta$ then choose $Z_i \subset X_i'$  such that $|Z_i| = \delta$ and then for every vertex $y \in Y_i',$ 
  $$|N_G(y) \cap Z_i | \geq |Z_i| - 2\sqrt{\xi}|X_i'| \geq (1- 2 \cdot \frac{2}{\alpha} \sqrt{\xi})|Z_i| \geq (1-\tau)\delta, $$
  If $|Y_i'| > \delta$ then the same is true for vertices $x \in X_i' $. 
  Hence if $|X_i'|, |Y_i'| \geq \delta$ then we can choose $Z_i \subset X_i', W_i \subset Y_i'$ such that $|Z_i| = |W_i| = \delta$ 
  and for any $x \in Z_i, y \in W_i$, $|N(x) \cap W_i|, |N(y) \cap Z_i| \geq (1 - \gamma)\delta,$ so by Lemma \ref{almost complete bipartite}, $G$ contains $L_{\delta}$.
  Since $|Y_i'| \leq \frac{2}{\alpha} |X_i' \cup X_i''|,  \tau \cdot \frac{2}{\alpha} \leq \frac{1}{300}$, if $|X_i'| < \delta$, $|Y_i'| \geq \delta$ then by Lemma \ref{almost complete bipartite - 1}, $G[X_i' \cup X_i'', Y_i']$ contains $L_{\delta}$.

  Now, we may assume that $\min\{|X_i'\cup X_i''|, |Y_i'\cup Y_i''|\} < \delta$ for all $i \in [q]$. 
  \begin{claim} \label{two matchings between different bipartitions}
  Let $i \in [q]$.
  If there exists $j \in [q]$ such that there exists  $Z_i \in \{X_i', Y_i'\}, Z_j \in \{X_j' \cup X_j'', Y_j' \cup Y_j''\}$ such that $E(Z_i, Z_j)$ has a matching of size 2, 
  then $G$ contains a $(n',r)$ weak ladder such that $n' -r \geq \delta$.
  \end{claim}
  \begin{proof}
  Without loss of generality, let $i = 1, j=2$ and $Z_1 = X_1', Z_2 = X_2' \cup X_2''$.
  Let $\{u_1,u_2\}, \{v_1,v_2\} \in E(X_1', X_2' \cup X_2'').$
  For $i \in [2]$, choose $e_i$ such that $u_i \in e_i$ and $e_i \cap Y_1' \neq \emptyset$, $e_i'$ such that $v_i \in e_i'$ and $e_i' \cap Y_2' \neq \emptyset$.
  For $i \in [2]$, by Lemma \ref{almost complete bipartite - 1},  $G[X_i' \cup X_i'' \cup Y_i']$ contains $L_{t}$ having $e_i,e_i'$ are in its first 4th rung where $t \geq \frac{2\delta}{3}$.
  By attaching these two ladders with $\{u_1,u_2\}, \{v_1,v_2\}$, we obtain a $(n',r)$-weak ladder such that $r \leq 10$ and $n' - r \geq (2t - 10) - 10 \geq \frac{4\delta}{3} - 20 \geq \delta$. 
  \end{proof}

  \begin{claim} \label{at most 1 edge in each bipartition}
  If there exists $i \in [q]$ such that $||G[X_i' \cup X_i'']|| + ||G[Y_i' \cup Y_i'']|| \geq 2$ then $G$ contains a $(n',r)-$weak ladder for some  $n' \geq \delta + r$.
  \end{claim}
  \begin{proof}
  Let $j \neq i$ and recall that   $V_i' = X_i' \cup X_i'' \cup Y_i'\cup Y_i''$, $V_j' = X_j' \cup X_j'' \cup Y_j' \cup Y_j''$. Since $G$ is 2-connected there are disjoint $V_i'-V_j'$-paths $P,Q$ such that $|P|+|Q|\leq \frac{10}{\alpha}$ from Fact \ref{short-paths}.
  Let $\{x_1\} = V(P) \cap V_i', \{x_2\} = V(Q) \cap V_i'  $ and $\{y_1\} = V(P) \cap V_j',\{y_2\} =  V(Q) \cap V_j' $.
  
  Let $\{z_1,z_2\} \in E(G[X_i' \cup X_i'']) \cup E(G[Y_i' \cup Y_i''])$ be such that $|\{z_1, z_2\} \cap \{x_1,x_2\}| \leq 1$, which is possible since $||G[X_i' \cup X_i'']|| + ||G[Y_i' \cup Y_i'']|| \geq 2$.
  We remove some vertices in $X_i'' \cup Y_i'' \setminus \{x_1,x_2,z_1,z_2\}$ and some vertices in $X_j'' \cup Y_j'' \setminus \{y_1,y_2\}$ so that $|X_i' \cup X_i''|=|Y_i' \cup Y_i''|, |X_j' \cup X_j''|=|Y_j' \cup Y_j''|$.

  For any $x \in \{x_1,x_2,z_1,z_2\}$, if $x \in X_i' \cup X_i''(Y_i' \cup Y_i'')$ 
  choose $x' \in N(x) \cap Y_i' (X_i')$, say $e(x) = \{x,x'\}$, then we have $E_0 = \cup_{x \in \{x_1,x_2,z_1,z_2\}}e(x) $ 
  such that $|E_0| \leq 4$ and for any $e \in E_0$, $|e \cap (X_i' \cup Y_i')| \geq 1$.
  Similarly, for any $y \in \{y_1,y_2\}$, if $y \in X_j' \cup X_j''(Y_j' \cup Y_j'')$ 
  choose $y' \in N(y) \cap Y_j' (X_j')$, say $e(y) = \{y,y'\}$, then we have $E_0' = \cup_{y \in \{y_1,y_2\} }e(y)$ such that $|E_0'| = 2$ and for any $e \in E_0'$, $|e \cap (X_j' \cup Y_j')| \geq 1$.
  Then by Lemma \ref{almost complete bipartite}, there exist ladders $L_{|X_i' \cup X_i''|}$ in $G[V_i']$ and $L_{|X_j' \cup X_j''|} $ in $G[V_j']$ 
  such that $E_0, E_0'$ are in those first $10$ rungs. 
  Since $|X_i' \cup X_i''| + |X_j' \cup X_j''| \geq \frac{4\delta}{3}$, and $20 + \frac{5}{\alpha} < \frac{\delta}{6}$, we obtain $(n',r)-$weak ladder 
  such that $n' \geq \frac{4\delta}{3} - 20$, $r \leq 20 + \frac{10}{\alpha}$ and so 
  $n' - r \geq \frac{4\delta}{3} - \frac{10}{\alpha} - 40 \geq \delta + \frac{\delta}{3} - \frac{30}{\alpha} \geq \delta. (\because n \geq \frac{90}{\alpha^2}.)$
  \end{proof}
  
  Now, we choose $i \in [q]$ such that $ \min\{|X_i'\cup X_i''|, |Y_i'\cup Y_i''|\}$ is maximum and we redistribute vertices from $V_j$, for $j \in [q] \setminus \{i\}$ as follows. Without loss of generality, let $|X_i'\cup X_i''|< \delta$. 
  If there exists $v \in V_j'$ such that $|N(v) \cap Y_i'| \geq 4\tau |Y_i|'$ then we move it to $X_i''$ until $|X_i'\cup X_i''|= \delta$. We apply the same process to $Y_i' \cup Y_i''$ if $|Y_i'\cup Y_i''|< \delta$.
  After the redistribution, if $\min\{|X_i'\cup X_i''|, |Y_i'\cup Y_i''|\}  = \delta$ then again Lemma \ref{almost complete bipartite} and Lemma \ref{almost complete bipartite - 1} imply existence of $L_{\delta}$.
  Thus assume $|X_i'| + |X_i''|$ is less than $\delta$ after redistribution. Since $||G[Y_i' \cup Y_i'']|| \leq 1$, at least $|Y_i'| - 2$ vertices $y \in Y_i'$ have a neighbor in $V(G)\setminus V_i'$.  
  Therefore for some $j \neq i$ and $Z_j'\in \{X_j'\cup X_j'', Y_j \cup Y_j''\}$, $|E_G(Y_i', Z_j')| \geq (|Y_i'|-2)/(2q-2) \geq \delta|Y_i'|/(3n)$. 
  If there is a matching of size two in $G[Y_i',Z_j']$, then by Claim \ref{two matchings between different bipartitions}, we obtain a $(n',r)$-weak-ladder with $n' \geq \delta + r$.
  Otherwise, there is a vertex $z \in Z_j'$ such that $|N_G(z) \cap Y_i'| \geq 4\tau |Y_i'|$  and then we can move $z$ to $X_i''$.
  \end{proof}
  
  Finally, we will prove the case when all the components are small.
  \begin{lemma}\label{small components}
  If every component $D$ of $R$ satisfies $|D| \leq (\delta/n + 3d_1)t$ then either $G$ contains disjoint cycles $C_{2n_1}, C_{2n_2}, \dots, C_{2n_l}$ for every $n_1, \dots, n_l\geq 2$ such that $\sum n_i =\delta$ or $\delta$ is even, $n_1=n_2=\frac{\delta}{2}$ and $G$ is one of the graphs from Example \ref{ex1}.
  \end{lemma}
  \begin{proof}
  Let $\xi = 6d_1/\alpha$, $\tau := \frac{100d_1}{\alpha^2}$. Note that $3\sqrt{\xi} \leq \tau \leq \frac{\alpha}{40} < \frac{1}{40}$.
  Since $d_1< \gamma/2$, there are at least three components. Indeed, otherwise
  $$|V| \leq 2(\delta/n+3d_1)n +\epsilon_1 n = 2\delta +(3d_1+\epsilon_1)n \leq n - (2\gamma -3d_1-\epsilon_1) <n.$$
  Let $q$ be a number of components.
  Let $V_D=\bigcup_{X \in D} X$ and let $G_D=G[V_D]$. Note that $\alpha n/2 \leq \delta -3d_1n \leq |V_D| \leq \delta +3d_1n \leq 2\delta$ and we have
  $$|E_G(V_D, V\setminus V_D)| \leq |D|\frac{n}{t} d_1n +\epsilon_1 n^2 \leq 2d_1\delta n.$$
  Thus
  $$|E(G_D)| \geq \frac{\delta|V_D|}{2} - 2d_1\delta n \geq {|V_D| \choose 2}- 4d_1\delta n\geq (1-\xi){|V_D|\choose 2}.$$
  
  Let $V_D'=\{v\in V_D| |N_G(v)\cap V_D| \geq (1-\sqrt{\xi})|V_D|\}$ and note that $|V_D'|\geq (1-2\sqrt{\xi})|V_D| \geq (1-3\sqrt{\xi})\delta$ and for any $v \in V_D'$, 
  $$|N_G(v) \cap V_D'| \geq (1-3\sqrt{\xi})|V_D'| \geq (1-\tau)|V_D'|. $$
  Move vertices from $V_D \setminus V_D'$ to $V_0$ to obtain $V_0'$. We have $|V_0'| \leq (\epsilon_1+ 2\sqrt{\xi})n\leq 3\sqrt{\xi}n$ and $|V_D'|\geq (1-2\sqrt{\xi})|V_D| \geq (1-3\sqrt{\xi})\delta$.
  Now we redistribute vertices from $V_0'$ as follows. Add $v$ from $V_0'$ to $V_D''$ if $|N(v)\cap V_D'| \geq 4\tau |V_D'|.$
  Since $|V_0'|\leq 3\sqrt{\xi}n \leq \delta/3$ and the number of components is at most $n/(\delta-3d_1n)\leq 2n/\delta$, so for every $v \in V_0'$, there exists a component $D$ such that $|N(v)\cap V_D'| \geq \frac{\alpha}{3} \delta \geq \frac{\alpha}{6} |V_D'|  \geq 4\tau |V_D'|$. 
  Let $V_D^*:= V_D'\cup V_D''$. Note that $|V_D''| \leq 3\sqrt{\xi}n \leq \tau |V_D'|$, which says $|V_D'| \geq (1-\tau)|V_D^*|$.
  Hence for any component $D$, $G[V_D^*]$ is $\tau$-complete.

  \begin{claim}\label{claim-matching}
  If $D_1,D_2$ are two components and there is a matching of size four between $V_{D_1}^*$ and $V_{D_2}^*$, then $G$ contains disjoint cycles $C_{2n_1}, C_{2n_2}, \dots, C_{2n_l}$ for every $n_1, \dots, n_l\geq 2$ such that $\sum n_i =\delta$.
  \end{claim}
  \begin{proof} 
  Let $D$ be a component which is different than $D_1$ and $D_2$.
  By Fact \ref{short-paths}, there exist two $V_D^*-(V_{D_1}^*\cup V_{D_2}^*)$-paths $P,Q$ which can contain vertices from at most two edges in the matching. 
  Let $u,v \in (V(P) \cup V(Q)) \cap V_D^*, x,y \in (V(P) \cup V(Q)) \cap (V_{D_1}^*\cup V_{D_2}^*)$ and let $\{x',y'\}, \{x'',y''\}$ be a matching in $E(V_{D_1}^*, V_{D_2}^* )$ such that $\{x',y',x'',y''\} \cap \{x,y\} = \emptyset$.
  Then we have two cases:
  \begin{itemize}
  \item $x,y$ are in a same component, without loss of generality, let $x,y \in V_{D_1}^*$.
  By applying Lemma \ref{almost complete} to each component, we obtain a ladder in each component.
  If $n_1$ is such that $n_1 > |V_D^*| - 7$ then $C_{2n_i}$ can be obtained by attaching a ladder in $G[V_D*]$ and some first rungs in a ladder in $G[V_{D_1}^*]$ with $P,Q$ and a parity vertex(if necessary) in $G[V_D^*]$.
  If $n_2$ is such that $n_2 > |V_{D_2}^*| - 7$ then $C_{2n_i}$ can be obtained by attaching a ladder in $G[V_{D_2}*]$ and some last rungs in a ladder in $G[V_{D_1}^*]$ with $\{x',x'' \}, \{y',y''\}$.
  Moreover, remaining small cycles can be obtained in a ladder remained in $G[V_{D_1}^*]$.
  Otherwise, the case is trivial.
  \item Let $x \in V_{D_1}^*, y \in V_{D_2}^*$.
  Since there is a matching of size four between $V_{D_1}^*$ and $V_{D_2}^*$, there is a matching $\{x''', y'''\}$ in $E(V_{D_1}^*, V_{D_2}^*)$ or remaining two matching $e_1, e_2$ are such that $x \in e_1, y \in e_2$, say $e_1 = \{x, y'''\}, e_2 = \{x''',y\}$.
  In both case, we obtain a ladder starting at $N(u), N(v)$ in $G[V(D^*)]$.
  In the first sub-case, we choose a ladder starting at $N(x), N(x')$ ending at $N(x''),N(x''')$  in $G[V_{D_1}^*]$ and a ladder starting at $N(y), N(y')$ ending at $N(y''),N(y''')$  in $G[V_{D_2}^*]$. 
  By attaching those three ladders with using parity vertex in an appropriate manner, we obtain a desired structure containing disjoint cycles.
  In the other case, we choose a ladder starting at at $N(x), N(x''')$ ending at $N(x'),N(x'')$  in $G[V_{D_1}^*]$ and a ladder starting at $N(y), N(y''')$ ending at $N(y'),N(y'')$ in $G[V_{D_2}^*]$. Similarly, we are done by attaching those three ladders.
  \end{itemize}
  \end{proof}
By Claim \ref{claim-matching}, we may assume that for any $i,j \in [q]$, $E(V_{D_i}^*, V_{D_j}^*)$ has a matching of at most 3.
Then we have another claim which is useful for the arguments follow.
  \begin{claim} \label{potential vertex to be moved}
  Let $D$ be a component. For any $X \subset V \setminus V_{D^*}$, 
  if $|\{ v \in V_{D^*} : N(v) \cap X \neq \emptyset \} | \geq \frac{|V_{D^*}|}{2} $ then there exists $x \in X$ such that $|N(v) \cap V_{D}'| \geq 4\tau |V_{D}'|$.
  \end{claim}
  \begin{proof}
  Let $X$ is a subset of $V \setminus V_{D^*}$ and assume that $|\{ v \in V_{D^*} : N(v) \cap X \neq \emptyset \} | \geq \frac{|V_{D^*}|}{2} $.
  Then there exists $i \in [q]$ and $Y \subset V_{D_i}^*$ such that 
  $$|\{ v \in V_{D^*} : N(v) \cap Y \neq \emptyset \} | \geq \frac{|V_{D^*}|}{2} \cdot \frac{1}{q} \geq \frac{\alpha |V_{D^*}|}{2}.$$
  So, there exists $v \in Y$ such that 
  $$ |N(v) \cap V_{D^*}| \geq \frac{\alpha |V_{D^*}|}{6}, $$
  which implies that 
  $$ |N(v) \cap V_{D'}| \geq 4 \tau |V_D'|. $$
  \end{proof}

  \begin{claim} \label{two components}
  Let $D_1^*,D_2^*$ be two components. If there exist two distinct vertices $x,y \in V_{D_1}^*$ such that $|N(x) \cap V_{D_2}'|, |N(y) \cap V_{D_2}'| \geq 4\tau|V_{D_2}'|$, then there exists a $(n', r)-$ weak ladder where $r \in \{1,2\}$, $n' + r =  \lfloor \frac{|V_{D_1}^*| + |V_{D_2}^*|}{2} \rfloor $.
  \end{claim}
  \begin{proof}
  Since $|N(x) \cap V_{D_2}'|, |N(y) \cap V_{D_2}'| > \tau|V_{D_2}'|$, there is $x' \in N(x) \cap V_{D_2}'$, so there exists $y' \in N(y) \cap V_{D_2}'$ such that $x' \sim y'$.
  If $|V_{D_1}^*|$ is even, then by Lemma \ref{almost complete} (2), $G[V_{D_2}^*]$ contains a ladder $L_{\lfloor\frac{|V_{D_2}^*|}{2}\rfloor}$ having $\{x',y'\}$ as its first rung and $G[V_{D_1}^*]$ contains a ladder ${L_{\frac{|V_{D_1}^*|}{2} - 1 }}$ starting at $N(x),N(y)$.
  By attaching those two ladder with $\{x,x'\} , \{y,y'\}$, we obtain a $(n',1)$- weak ladder where $n' = \frac{|V_{D_1}^*|}{2} - 1  + \lfloor \frac{|V_{D_2}^*|}{2}\rfloor  = \lfloor \frac{|V_{D_1}^*| + |V_{D_2}^*|}{2} \rfloor - 1.$
  Now, suppose that $|V_{D_1}^*|$ is odd. 
  If $\{x,y\} \cap V_{D_1}'' \neq \emptyset$, without loss of generality, $x \in V_{D_1}''$, 
  then by Fact \ref{almost complete fact 1}, $G[V_{D_1}^* \setminus \{x \}]$ is $\tau$-complete.
  Since  $|N(x) \cap V_{D_2}'| \geq 4 \tau|V_{D_2}'|$, 
  $G[V_{D_2}^* \cup \{x \} ]$ is $\tau$-complete. Since there exists $x'' \in V_{D_1}^* \cap N(x)$ such that $x'' \neq y$,
  by Lemma \ref{almost complete} (2), there exists a $L_{\frac{|V_{D_1}^*| - 1}{2} - 1}$ the first rung $e_1 = \{v_1,v_2\}$ of which is such that $v_1 \in N(x''), v_2 \in N(y)$ in $G[V_{D_1}^* \setminus \{x\}]$, 
  and there exists a $L_{\lfloor \frac{|V_{D_2}^*| + 1}{2}\rfloor -1 }$ the first rung $e_2 = \{v_1',v_2'\}$ of which is such that $v_1' \in N(x), v_2' \in N(y')$ in $G[V_{D_2}^* \cup \{x\}]$,
  so by attaching two ladders with $\{x,x''\}$ and $\{y,y'\}$,  we obtain a $(\lfloor \frac{|V_{D_1}^*| + |V_{D_2}^*|}{2} \rfloor  -2,2)$-weak ladder.
  If $\{x,y\} \subset V_{D_1}'$ then there exists $z \in N(x) \cap N(y) \cap V_{D_1}^*$. 
  Since $x',y' \in V_{D_2}'$, there exists $z' \in N(x') \cap N(y') \cap V_{D_2}^*$.
  By Lemma \ref{almost complete} (4), there exists a $L_{\frac{|V_{D_1}^*| - 1}{2}}$ the first rung of which is $\{z,y \}$ in $G[V_{D_1}^* \setminus \{x\}]$, 
  and there exists a $L_{\lfloor \frac{|V_{D_2}^*|-1}{2} \rfloor}$ the first rung of which is $\{z',y' \}$ in $G[V_{D_2}^* \setminus \{x'\}]$.
  By attaching two ladders with $\{x,x'\} , \{y,y'\} $, we obtain a $(\lfloor \frac{|V_{D_1}^*| + |V_{D_2}^*|}{2} \rfloor  -1,1)$-weak ladder.
  \end{proof}
  
  \begin{claim} \label{two components - 1}
  If $D_1^*,D_2^*$ are two components such that $|V_{D_1}^*| + |V_{D_2}^*| \geq 2K + 1$,
  then there exists a $(n', c)-$ weak ladder where $n' \geq K - 2,  2 \leq c \leq \frac{7}{\alpha}$.
  \end{claim}
  \begin{proof}
  We can always delete a vertex from $V_{D_1}^*$ and so we may assume that 
  $|V_{D_1}^*| + |V_{D_2}^*| = 2K + 1$.
  Then exactly one of the terms $|V_{D_i}^*|$ is odd, and we have $\lfloor \frac{|V_{D_1}^*|}{2} \rfloor + \lfloor \frac{|V_{D_2}^*|}{2} \rfloor = K$.
  By Fact \ref{short-paths}, there exist two disjoint path $P,Q$ between $V_{D_1}^*,V_{D_2}^*$ such that $|V(P)| + |V(Q)| \leq \frac{10}{\alpha}$.
  Let $x_i$ be the endpoints of $P$ and $y_i$ be the ends of $Q$ where $x_i,y_i \in V_{D_i}^* , i \in [2]$. If $|V(P)|+|V(Q)|$ is even then by applying Lemma \ref{almost complete} (2),
  we obtain a $L_{\lfloor \frac{|V_{D_1}^*|-2 }{2}\rfloor}$ in $G[V_{D_1}^*]$ and $L_{\lfloor \frac{|D_2^*|-2}{2} \rfloor}$ in $G[V_{D_2}^*]$ which start at $N(x_1),N(y_1)$, $N(x_2),N(y_2)$, respectively. 
  By attaching these ladders, we obtain a $(n',c)-$weak ladder where $n' = \lfloor \frac{|V_{D_1}^*|-2}{2} \rfloor + \lfloor \frac{|V_{D_2}^*|-2}{2} \rfloor \geq K - 2$ and $2 \leq c \leq \frac{4 + \frac{10}{\alpha}}{2} \leq \frac{6}{\alpha}$.
  Otherwise, assume that $|V(P)| + |V(Q)|$ is odd.
  If $|V_{D_i}^*|$ is odd then $|V_{D_3-i}^*|$ is even, and we apply Lemma \ref{almost complete} (3) to $G[V_{D_i}^*]$ and Lemma \ref{almost complete} (2) to $G[V_{D_3-i}^*]$. 
  Then by attaching those two ladders, we obtain a $(n',c)$-weak ladder where $n' \geq K-2$ and $3 \leq c \leq \frac{5 + \frac{10}{\alpha}}{2} \leq \frac{7}{\alpha}$.
  \end{proof}
  
  Now, we move vertices between components to obtain, if possible, components of larger size. 
  Let $D$ be a component and suppose $|V_{D}^*|\leq \delta$. Then every vertex $v \in V_D'$ has a neighbor in $V\setminus V_D^*$. Thus $|E(V_D', V\setminus V_D^*)| \geq |V_D'|$.
  If there is a matching of size $8n/\delta$, then for some component $F \neq D$
  there is a matching of size four between $V_D'$ and $V_F^*$, and we are done by Claim \ref{claim-matching}.
  Hence there is a vertex $v \in V \setminus V_D^*$ such that $|N(v) \cap V_D^*| > \delta |V_D'|/8n$, so $|N(v) \cap V_D'| \geq 4\tau |V_D'|$, then we can {\it move} $v$ to $V_D''$.
  Thus we may assume that there is a component $D$ such that $|V_D^*| \geq \delta + 1$.
  We will now move vertices between components. To avoid introducing new notation, we will use $D_i^*$ to refer to the $i$th component after moving vertices from and to $D_i^*$.  
  Move vertices so that after renumbering of components we have $|V_{D_1}^*| \geq |V_{D_2}^*| \geq \dots \geq |V_{D_q}^*|$ and for any $k \in [q]$,  $\sum_{i=1}^{k} |V_{D_i}^*|$ is as big as possible
  and subject to this, for any $i,j \in [q] \setminus \{1\}$, $|E(V_{D_1}^*, V_{D_i}^*)| \leq |E(V_{D_1}^*, V_{D_j}^*)|$ where $i < j$.
  Note that if $i < j$ and $|V_{D_i}^*| = |V_{D_j}^*|$ then $|E(V_{D_1}^*, V_{D_i}^*)| \leq |E(V_{D_1}^*, V_{D_j}^*)|$.
  If $|V_{D_1}^*| \geq \delta + 14d_1 n$ then we stop moving any vertices.
  Hence the natural first case is that $|V_{D_1}^*|\geq \delta + 11d_1n$. Since $q-1 \geq 2$, there exists $i \in [q] \setminus \{1\}$ such that 
  at most $7d_1n$ vertices in the original $V_{D_i}^*$ were moved to $V_{D_1}^*$, $G[V_{D_i}^*]$ is $2\tau$-complete and
  $$ |V_{D_1}^*| + |V_{D_i}^*| \geq (\delta + 11d_1n) + (\delta - 10d_1n) = 2\delta + d_1n. $$ 
  By Claim \ref{two components - 1}, $G[V_{D_1}^* \cup V_{D_2}^*]$ contains a $(n',c)$-weak ladder where $n' \geq \delta + \frac{d_1n}{2} - 2$, $2 \leq c \leq \frac{7}{\alpha}$.
  Since $\frac{d_1 n}{2} \geq \frac{8}{\alpha} \geq \frac{7}{\alpha}+2$, by Lemma \ref{weak ladder criteria - 1},  $G[V_{D_1}^* \cup V_{D_2}^*]$ contains all disjoint cycles. 
  So we may assume that all possible moves terminate and 
  $$\delta + 11d_1n > |V_{D_1}^*| \geq |V_{D_2^*}| \geq \dots \geq |V_{D_{q}^*}|.$$ 
  Note that for any $i \in [3],j \in [q]$ such that $j>i$, there is no $v \in V_{D_j}^*$ such that $|N(v) \cap V_{D_i}'| \geq \delta |V_{D_i}'| / (16n)$, since otherwise, $v$ can be moved to $V_{D_i}''$.
  There are at most $(q-1)\cdot 14d_1n$ vertices in each original components moved to other components, 
  but since $(q-1)\cdot 14d_1n \leq \frac{50d_1n}{\alpha}$, for $i \geq 2$, $G[V_{D_i}^*]$ is $2\tau$-complete, because $\tau \geq \frac{100d_1}{\alpha^2}$.
  
  We will continue the analysis based on the size of $V_{D_1}^*$.  Suppose $|V_{D_1}^*| = \delta + c_1$ where $c_1 \geq 1$. 
  \begin{claim} \label{delta + 1} 
  If $|V_{D_1}^*| = \delta + 1$ then for any $u,v \in V_{D_1}^*$, $N(u) \cap N(v) \cap V_{D_1}^* \neq \emptyset$.
  \end{claim}
  \begin{proof}
  Suppose not and let $u,v \in V_{D_1}^*$ be such that $N(u) \cap N(v) \cap V_{D_1}^* = \emptyset$.
  Then for any $x \in V_{D_1}^* \setminus \{u,v\}$, $|N(x) \cap \{u,v\}| \leq 1$, so $|N(x) \cap V_{D_1}^*| \leq |V_{D_1}^*| - 2 = \delta - 1$, therefore 
  $$|E(V_{D_1}^*, V \setminus V_{D_1}^*)| \geq  |V_{D_1}^*| - 2 \geq \frac{|V_{D_1}^*|}{2},$$
  by Claim \ref{potential vertex to be moved}, there exists $z \in V \setminus V_{D_1}^*$ which can be moved to $V_{D_1}^*$, a contradiction.
  \end{proof}

  First, we consider the case that $|V_{D_2}^*| \geq \delta$. 
  By Fact \ref{short-paths}, there are $V_{D_1}^* - V_{D_2}^*$ paths $P_1,P_2$ such that $|V(P_1)|+|V(P_2)| \leq \frac{10}{\alpha}$.
  Denote by $u_i, v_i$ the endpoints of $P_i$ such that $u_i \in V_{D_1}^*, v_i \in V_{D_2}^*$.
  Since $|V_{D_3}^* \cap (V(P_1) \cup V(P_2)| \leq \frac{10}{\alpha}$, $G[V_{D_3}^* \setminus (V(P_1) \cup V(P_2))]$ is $2\tau$-complete (because $\tau |V_{D_3}^*| \geq \tau \cdot \frac{\delta}{2} \geq \frac{20}{\alpha} $), so by Lemma \ref{almost complete},
  $G[V_{D_3}^* \setminus (V(P_1) \cup V(P_2))]$ contains a $L_{n''}$ where $n' \geq \frac{\delta}{3}$.
    
  Since $|V_{D_1}^*| + |V_{D_2}^*| \geq 2\delta + 1$, by Claim \ref{two components - 1}, $G[V_{D_1}^* \cup V_{D_2}^* \cup V(P_1) \cup V(P_2)]$ contains a $(\delta -2,k)$-weak ladder where $k \geq 2$.
  By Corollary \ref{weak ladder criteria - 2}, we may assume that either $n_1 = \lfloor \frac{\delta}{2} \rfloor, n_2 = \lceil \frac{\delta}{2} \rceil$ or $n_1 = \lfloor \frac{\delta - 1}{2} \rfloor, n_2 = \lceil \frac{\delta+1}{2} \rceil$.
  If $n_1 = \lfloor \frac{\delta}{2} \rfloor, n_2 = \lceil \frac{\delta}{2} \rceil$ then $G[V_{D_1}^*]$ contains $C_{2n_2}$ and $G[V_{D_2}^*]$ contains $C_{2n_1}$.
  Otherwise, let $n_1 = \lfloor \frac{\delta - 1}{2} \rfloor, n_2 = \lceil \frac{\delta+1}{2} \rceil$.
  If $\delta$ is odd then $G[V_{D_1}^*]$ contains $C_{2n_2}$ and $G[V_{D_2}^*]$ contains $C_{2n_1}$, so let $\delta$ be even.
  If $u_1 \sim u_2$ then by Lemma \ref{almost complete} (5), $G[V_{D_2}^* \cup \{u_1,u_2\}]$ contains $C_{2n_2}$ and by Lemma \ref{almost complete} (2), $G[V_{D_1}^* \setminus \{u_1,u_2\}]$ contains $C_{2n_1}$.
  If $u_1 \nsim u_2$ then by Claim \ref{delta + 1}, there exists $u_3 \in N(u_1) \cap N(u_2) \cap V_{D_1}^*$, and then $G[V_{D_2}^* \cup \{u_1,u_2,u_3\}]$ contains $C_{2n_2}$ and $G[V_{D_1}^* \setminus \{u_1,u_2,u_3\}]$ contains $C_{2n_1}$.

  Now, we assume that $|V_{D_2}^*| < \delta$, so let $|V_{D_2}^*| = \delta - c_2$ where $c_2 \geq 1$. By Claim \ref{claim-matching}, $|E(V_{D_1}^*, V_{D_2}^*) | \leq 3 |V_{D_2}^*| $, and so 
  $$ |E(V_{D_2}^*, V \setminus (V_{D_1}^* \cup V_{D_2}^*) )| \geq  (c_2+1 -3) |V_{D_2}^*| \geq (c_2 -2) |V_{D_2}^*| .$$ 
  By Claim \ref{potential vertex to be moved}, $|E(V_{D_2}^*, V \setminus (V_{D_1}^* \cup V_{D_2}^*))| < \frac{|V_{D_2}^*|}{2}$, so $c_2 \in [2]$ 
  and there exist two distinct vertices $x,y \in V_{D_1}^*$ such that $|N(x) \cap V_{D_2}^*|+|N(y) \cap V_{D_2}^*| \geq \frac{3|V_{D_2}^*|}{2}$. It implies that $|N(x) \cap N(y) \cap V_{D_2}^*| \geq \frac{|V_{D_2}^*|}{2}$  and $|N(x) \cap V_{D_2}'|, |N(y) \cap V_{D_2}'| \geq 4 \tau |V_{D_2}'|$.
  By Claim \ref{two components}, there exists a $(n', k)-$ weak ladder such that $k \in [2]$ and $n' + k \geq  \lfloor \frac{|V_{D_1}^*| + |V_{D_2}^*|}{2} \rfloor$.

  We have two cases.
  \begin{itemize}
  \item {\bf Case 1: $c_1 \geq c_2$}. Note that $|V_{D_1}^*| + |V_{D_2}^*| = (\delta + c_1) + (\delta - c_2) \geq 2\delta$. 
  Hence $G[V_{D_1}^* \cup V_{D_2}^*]$ contains either a $(\delta-1,1)$-weak ladder or $(\delta-2,2)$-weak-ladder.
  Since $G[V_{D_3}^*]$ contains $L_{n''}$ where $n'' \geq \frac{\delta}{3}$, by Corollary \ref{weak ladder criteria - 2},
  it suffices to show that $G$ contains disjoint $ C_{2\lfloor \frac{\delta}{2} \rfloor  } $,$ C_{2 \lceil \frac{\delta}{2} \rceil  }$ or disjoint $C_{2\lfloor \frac{\delta-1}{2} \rfloor }$,$ C_{2 \lceil \frac{\delta + 1}{2} \rceil  }$.
  We can choose $Z \subset \{x,y\}$ so that $G[V_{D_1}^* \setminus Z]$ contains $ C_{2\lfloor \frac{\delta}{2} \rfloor  }$ and $G[V_{D_2}^* \cup Z]$ contains $C_{2 \lceil \frac{\delta}{2} \rceil  }$.
  Since $N(x) \cap N(y) \cap V_{D_2}^* \neq \emptyset$, we can choose $Z \subset N(x) \cap N(y) \cap V_{D_2}^*$ so that  $G[V_{D_1}^* \cup Z]$ contains $ C_{2 \lceil \frac{\delta + 1}{2} \rceil  }$ and $G[V_{D_2}^* \setminus Z]$ contains $C_{2\lfloor \frac{\delta-1}{2} \rfloor} $.
  \item {\bf Case 2: $c_1 < c_2$}. Then $c_1 = 1, c_2 = 2$. Since $|E(V_{D_2}^*, V \setminus (V_{D_1}^* \cup V_{D_2}^*))| < \frac{|V_{D_2}^*|}{2}$,
  $|E(V_{D_2}^*, V_{D_1}^*)| \geq \frac{5|V_{D_2}^*|}{2}$, so again by Claim \ref{claim-matching}, 
  there exist $x,y,z \in V_{D_1}^*$ such that $E(V_{D_2}^*, V_{D_1}^*)= E(V_{D_2}^*, \{x,y,z\})$ and 
  for any $u \in \{x,y,z\}$, $|E(V_{D_2}^*,\{u\})| \geq \frac{|V_{D_2}^*|}{2}$.
  Note that $G[V_{D_1}^* \cup V_{D_2}^*]$ contains a $(n',1)$-weak ladder where $n' \geq \delta - 2$ and $G[V_{D_3}^*]$ contains $L_{n''}$ where $n'' = \lfloor \frac{|V_{D_3}^*|}{2} \rfloor$.
  If there exists $n_i$ such that $2 < n_i \leq n''$  then $G[V_{D_3}^*]$ contains $C_{2n_i}$, 
  and by Lemma \ref{weak ladder criteria - 1}, $G[V_{D_1}^* \cup V_{D_2}^*]$ contains remaining disjoint cycles.  If for every $i$, $n_i = 2$, then $G$ obviously contains all $C_{2n_i}$ for $i \in [l]$.
  
  Hence we may assume that $l = 2$ and $|V_{D_3}^*| < 2n_1 \leq 2n_2 $. 
  \begin{claim} \label{matching at most 1}
  For any $i,j \in [q] \setminus \{1\}$, the size of a maximum matching in $E(V_{D_i}^*, V_{D_j}^* )$ is at most one.
  \end{claim}
  \begin{proof}
  Suppose to a contrary that there exist $i,j \in [q] \setminus \{1\}$ such that $E(V_{D_i}^*, V_{D_j}^* )$ contains a matching of size at least two.
  If $2 \notin \{i,j\}$ then by Corollary \ref{attaching two components}, $G[V_{D_i}^* \cup V_{D_j}^*]$ contains a $(n',2)-$weak ladder where $n' \geq \lfloor \frac{|D_{V_i}^*|}{2} \rfloor + \lfloor \frac{|D_{V_j}^*|}{2} \rfloor - 2 \geq \frac{3\delta}{4},$ and then $C_{2n_2} \subset G[V_{D_1}^* \cup V_{D_2}^*], C_{2n_1} \subset G[V_{D_i}^* \cup V_{D_j}^*]$.
  Otherwise, without loss of generality, $i = 2$, $j > 2$. Let $\{e_1,e_2\}$ be a matching in $E(V_{D_2}^*, V_{D_j}^*)$.
  Let $x' \in N(x) \cap V_{D_2}', y' \in N(y) \cap V_{D_2}'$ such that $x' \sim y'$ and $\{x',y'\} \cap (e_1 \cup e_2) = \emptyset$. 
  Then $G[V_{D_1}^* \cup \{x',y'\}]$ contains $C_{2n_1}$. Since $G[V_{D_2}^* \setminus \{x',y'\}]$ is $2\tau$-complete, 
  $G[V_{D_2}^* \cup V_{D_j}^* \setminus \{x',y'\}]$ contains $C_{2n_2}$.
  \end{proof}

  If $|V_{D_3}^*| \leq \delta - 4$, then by Claim \ref{potential vertex to be moved}, $|E(V_{D_3}^*, V_{D_1}^* \cup V_{D_2}^*)| > \frac{9|V_{D_3}^*|}{2}$ 
  and then $|E(V_{D_3}^*, V_{D_2}^*)| \geq \frac{3|V_{D_3}^*|}{2}$, and we are done by Corollary \ref{attaching two components}.
  So we may assume that $|V_{D_3}^*| \in \{ \delta - 2, \delta - 3 \}$ which leads to two sub-cases.

  \begin{itemize}
    \item $|V_{D_3}^*| = \delta - 3$. By Claim \ref{potential vertex to be moved}, $|E(V_{D_3}^*, V_{D_1}^* \cup V_{D_2}^*)| > \frac{7|V_{D_3}^*|}{2}$.
    By Claim \ref{claim-matching}, $|E(V_{D_3}^*, V_{D_1}^* )| \leq 3|V_{D_3}^*|$, so by Claim \ref{matching at most 1}, there exists $w \in V_{D_2}^*$ such that $ \frac{|V_{D_3}^*|}{2} < |E(V_{D_2}^*,V_{D_3}^*)|= |E(\{w\}, V_{D_3}^*)| \leq | V_{D_3}^*|. $ 
    Hence $|E(V_{D_3}^*, V_{D_1}^* )| > \frac{5|V_{D_3}^*| }{2}$, and then there exist $x_1,y_1 \in V_{D_1}^*$ such that $|N(x_1) \cap V_{D_3}^*|,|N(y_1) \cap V_{D_3}^*| >  \frac{|V_{D_3}^*|}{2}. $ 
In addition, there is a vertex $z\in V_{D_1}^*\setminus \{x_1,y_1\}$ such that $|N(z)\cap V_{D_3}^*|> \frac{|V_{D_3}^*|}{2}$.  Choose $z' \in N(z) \cap N(w) \cap V_{D_3}^*$. Then $G[V_{D_2}^* \cup \{z,z'\}]$ contains $C_{2n_1}$
    and $G[V_{D_1}^* \cup V_{D_3}^* \setminus \{z,z' \}]$ contains  $C_{2n_2}$.
    \item  $|V_{D_3}^*| = \delta - 2$. Note that we may assume
    $$n_1 = \lfloor \frac{\delta}{2} \rfloor, n_2 = \lceil \frac{\delta}{2} \rceil, $$
  as otherwise $2n_1\leq \delta -2$, and $G[V_{D_3}^*]$ contains $C_{2n_1}$, $G[V_{D_1}^* \cup V_{D_2}^*]$ contains $C_{2n_2}$.
    \begin{claim} \label{delta - 2}
    For any $i \in [q] \setminus \{1\}$ such that $|V_{D_i}^*| = \delta -2$,
    $|E(V_{D_1}^*, V_{D_i}^* )| \geq \frac{5|V_{D_i}|}{2} $ and $E(V_{D_1}^*, V_{D_i}^*) = E(\{x,y,z\}, V_{D_i}^*) $.  
    \end{claim}
    \begin{proof}
    Let $i \in [q] \setminus \{1\}$ be such that $|V_{D_i}^*| = \delta -2$.
    Since $|V_{D_i}^*| = |V_{D_2}^*|$, by the redistribution process, 
    $$ |E(V_{D_1}^*, V_{D_i}^*)| \geq |E(V_{D_1}^*, V_{D_2}^*)| \geq \frac{5|V_{D_2}^*|}{2} = \frac{5|V_{D_i}^*|}{2}.$$
    There exists $x_1,y_1,z_1 \in V_{D_1}^*$ such that $E(V_{D_1}^*, V_{D_i}^*) = E(\{x_1,y_1,z_1\}, V_{D_i}^*) $ and
    $|N(x_1) \cap V_{D_i}^*|, |N(y_1) \cap V_{D_i}^*|, |N(z_1) \cap V_{D_i}^*| \geq \frac{|V_{D_i}^*|}{2}$,
    so $|N(x_1 \cap V_{D_i}')|, |N(y_1 \cap V_{D_i}')|, |N(z_1 \cap V_{D_i}')| \geq 4\tau |V_{D_i}'|$, and then for any $U \subset \{x_1,y_1,z_1\}, G[U \cup V_{D_i}^*] $ is $\tau$-complete.
    If $\{x_1,y_1,z_1\} \neq \{x,y,z\} $, w.l.o.g, $z_1 \notin \{x,y,z\}$,
     then $G[\{y_1,z_1 \} \cup V_{D_i}^*] $ contains $C_{2n_1}$ and $G[V_{D_2}^* \cup V_{D_1}^* \setminus \{y_1,z_1\} ]$ contains $C_{2n_2}$.
    \end{proof}

    \begin{claim} \label{every component is of size delta -2}
    For any $i \in [q] \setminus \{1\}$,  $|V_{D_1}^*| = \delta - 2.$ 
    \end{claim}
    \begin{proof}
    Suppose not and choose $i \in [q] \setminus \{1\}$ such that $|V_{D_i}^*| \leq \delta - 3$, and subject to this, $i$ is the smallest,
    i.e, for any $i' \in [i-1] \setminus \{1\}$, $|V_{D_i'}^*| = \delta - 2$. Note that $i \geq 4$.
    First, assume that there exists $i_1,i_2 \in [i-1] \setminus \{1\}$ such that 
    there exists $y_1 \in V_{D_{i_1}}^*$, $y_2 \in V_{D_{i_2}}^*$ such that $|N(y_1) \cap V_{D_i}^*|, |N(y_2) \cap V_{D_i}^*| > 0$.
    Let $Q$ be a $N(y_1)-N(y_2)$ path in $G[V_{D_i}^*]$.
    Since $|V_{D_{i_1}}^*|,|V_{D_{i_2}}^*| = \delta -2$, by Fact \ref{delta - 2},
    $$|E(V_{D_{i_1}}^*, \{x,y,z\} )|, |E(V_{D_{i_2}}^*, \{x,y,z\})| \geq \frac{5(\delta-2)}{2}.$$
    Choose $x' \in N(x) \cap V_{D_{i_1}^*}$ such that $x' \sim y_1$ and $x'' \in N(x) \cap V_{D_{i_2}^*}$ such that $x'' \neq y_2$.
    Then $G[ \{x\} \cup \{x',y_1\} \cup V(Q) \cup V_{D_{i_2}}]$ contains $C_{2n_1}$ and 
    $G[V_{D_1}^* \cup V_{D_{i_1}}^* \setminus \{x,x',y_1\}]$ contains $C_{2n_2}$.
    Hence $|E(V_{D_1}^*, V_{D_i}^*)| \geq \frac{5|V_{D_i}^*|}{2}$ and there exists $i' \in [i-1] \setminus \{1\}$ such that $|E(V_{D_{i'}}^*, V_{D_i}^*)| \geq \frac{|V_{D_i}^*|}{2}$ 
    and then there exists $V_{D_i}^* - V_{D_{i'}}^*$ path $Q$ such that $V(\mathring{Q}) \subset V_{D_1}^* $ and $|V(\mathring{Q}) \cap \{x,y,z\}| = 1 $.
    Similarly, we can find $C_{2n_1}$ in $G[V_{D_{i'} }^* \cup V_{D_i}^* \cup V(Q)]$ and $C_{2n_2}$ in $G[V_{D_1}^* \cup V_{D_j}^* \setminus V(Q)]$ where $j \in \{2,3\} \setminus \{i'\}$.
    \end{proof}
  
    Finally, suppose there exists $i,j(>i) \in [q] \setminus \{1\}$ such that $E(V_{D_i}^*, V_{D_j}^*) \neq \emptyset$. Let $e^* \in E(V_{D_i}^*, V_{D_j}^*)$.
    By Claim \ref{delta - 2}, $|N(x) \cap V_{D_i}^*|, |N(x) \cap V_{D_j}^*| \geq \frac{\delta-2}{2}$,
    so we can choose $x' \in (N(x) \cap V_{D_i}^*) \setminus e^*, x'' \in (N(x) \cap V_{D_j}^*) \setminus e^*$. 
    Then $G[V_{D_i}^* \cup V_{D_j}^* \cup \{x\}]$ contains $C_{2n_2}$ and $G[V_{D_1}^* \setminus \{x\}]$ contains $C_{2n_1}$.
    Therefore, for any $i,j \in [q] \setminus \{i\}$, $E(V_{D_i}^*, V_{D_j}^*) = \emptyset,$ which implies that 
    $G$ is a graph from Example \ref{ex1}.
  \end{itemize}

  \end{itemize}

  \end{proof}
  We can now finish the proof. If $R$ is connected then by Lemma \ref{done-or-extremal-claim}, \ref{weak ladder criteria - 1}, $G$ contains cycles $C_{2n_1}, \dots, C_{2n_l}$ or there exists a set $V' \subset V$ with $|V'| \geq (1-\delta/n - \beta)n$, such that all but at most $4 \beta n$ vertices $v\in V'$ have
  $|N_{G'}(v)|\le \beta n$ where $G' = G[V']$. If $R$ is disconnected and there is a component which is non-bipartite, then we are done by Lemma \ref{no bipartite},\ref{small components}, and \ref{weak ladder criteria - 1}, and if all components are bipartite, then $G$ has  $C_{2n_1}, \dots, C_{2n_l}$ by Lemma \ref{all bipartite}, \ref{weak ladder criteria - 1}.
  \end{proof}
  \section{The second non-extremal case}\label{sec:nonextr2}
  In this section we will show that if $G$ is non-extremal and $\delta(G)\geq (1/2 -\gamma)n$ for small enough $\gamma$, then $G$ contains disjoint cycles $C_{2n_1}, \dots, C_{2n_l}$.
  
  \begin{theorem}\label{non-extremal two}
    There exists $\gamma>0$ and $N$ such that for every 2-connected graph $G$ on $n\geq N$ vertices with $(1/2-\gamma)n \leq \delta(G) < n/2-1$, $G$ contains disjoint cycles $C_{2n_1}, \dots, C_{2n_l}$ for every $n_1,\dots, n_l$ where $n_i\geq 2$ and $n_1+\cdots + n_l= \delta(G)$ or $G$ is $\beta$-extremal for some $\beta=\beta(\gamma)$ such that $\beta\rightarrow 0$ as $\gamma \rightarrow 0$. In addition, if $G$ is not $\beta$-extremal and $n/2-1\leq\delta(G)\leq n/2$, then $G$ contains a cycle on $2\delta(G)$ vertices.
  \end{theorem}
  \begin{proof}We will use the same strategy as in the proof of Theorem \ref{non extremal main}. The first part of the proof is very similar to an argument from \cite{CK} and we only outline the main idea.
  Consider the reduced graph $R$ as in the proof of Theorem \ref{non extremal main}.
  
  First suppose $R$ is connected. We will use the procedure from \cite{CK} to show that either $G$ has a ladder on at least $n-1$ vertices or $G$ is $\beta$-extremal. Since $R$ is connected and
  $\delta(R)\geq (\delta/n -2d_1)t\geq (1/2- 2\gamma)t$, there is a path in $R$, $P=U_1V_1, \dots, U_sV_s$ where $s\geq (1/2-3\gamma)t$. As in \cite{CK} we move one vertex from $U_1$ to $U_s$, and the clusters in $R$ which are not on $P$ to $V_0$ so that $|V_0| \leq 7\gamma n$ and redistribute vertices from $V_0$ using the following procedure from \cite{CK}.
  Let $\xi, \sigma$ be two constants. The procedure is executed twice with different values of $\xi$ and $\sigma$. Distribute two vertices at a time and assign them to $U_i, V_j$ so that for every $i$, $|U_i| -|V_i|$ is constant, the number of vertices assigned to $U_i$ and $V_j$ is at most $O(\xi n/k)$, and if $x$ is assigned to $U_i$ ($V_j$), then $|N_G(x) \cap V_i| \geq \sigma n/k$ ($|N_G(x)\cap U_j| \geq \sigma n/k$).
  Let $Q$ denote the set of clusters $X$ such that $\xi n/k$ vertices have been assigned to $X$. We have $|Q| \leq 7\gamma k/\xi$.
  For $X\in \{U_i, V_i\}$, let $X^*$ be such that $\{X^*, X\}=\{U_i,V_i\}$. For a vertex $z$ let $N_z=\{X \in V(P)\setminus Q| |N_G(z)\cap X^*|\geq \sigma n/k\}$ and $N^*_z=\{X^*| X\in N_z\}$.
  
  Take $x,y$ from $V_0$, and choose $X, Y$ such that $X, X^*, Y, Y^*$ are not in $Q$, and either $N_x^*\cap N_y\neq \emptyset$ or  $N_x^*\cap N_y= \emptyset$ but $\exists_{X\in N_x, Y\in N_y} |E_G(X,Y)|\geq 2\sigma n^2/k^2$.
  The argument from \cite{CK} shows that either $G$ has a ladder on $2\lfloor n/2 \rfloor$ vertices or the algorithm fails. We will show that if it fails, then $G$ is $\beta$-extremal for some $\beta>0$.
  Since $|Q| \leq 7\gamma k/\xi$ and $|V_0| \leq 7\gamma n$, using the fact that $\delta(G)\geq (1/2-\gamma)n$, we have $|N_x| \geq \left(\frac{1}{2}- \frac{10\gamma}{\xi}\right)k$.
  If $N_x^*\cap N_y \neq \emptyset$, then we assign $x$ to $X\in N_x$ and $y$ to $X^*$ for some $X$ such that $X^*\in N_x^* \cap N_y$. Otherwise $N_x^*\cap N_y =\emptyset$ (and so $N_x$ and $N_y$ are almost identical). If there is $X\in N_x$ and $Y\in N_y$ such
  that  $|E(X,Y)|\geq 2\sigma n^2/k^2$, then assign $x$ to $X$, $y$ to $Y$ and a vertex $y'\in Y$ such that $|N_G(y)\cap X| \geq \sigma n/k$ to $X^*$. Otherwise $G$ is $\beta$-extremal for some $\beta>0$.
  
  We can now assume that $R$ is disconnected so it has two components $D_1,D_2$. Although slightly different arguments will be needed, we will reuse some parts of the proof of Lemma \ref{small components}.
  As in the proof of Theorem \ref{non extremal main}, we have $\delta(R)\geq (\delta/n -2d_1)t \geq (1/2-3d_1)t$. We set $\xi:= 12d_1$, $\tau = 400d_1$ and for a component $D$ define $V_D'=\{v||N_G(v)\cap V_D| \geq (1-\sqrt{\xi})|V_D|\}$ where $V_D$ is the set of vertices in clusters from $D$.
  As in the case of the proof of Lemma \ref{small components}, we have $|E(G_D)| \geq (1- \xi)\binom{|V_D|}{2}$ and similarly we also have $|V_D\setminus V_D'|\leq 2\sqrt{\xi}|V_D|$. We move vertices from $V_D \setminus V_D'$ to $V_0$ and 
  redistribute them to obtain $V_D''$ consisting of those vertices $v\in V_0$ for which $|N_G(v) \cap V_D'|\geq \frac{\delta}{6} \geq \tau |V_D'|$, and set $V_D^*:=V_D'\cup V_D''$. 
  We have $V_{D_1}^*\cup V_{D_2}^*= V(G)$ and $G[V_{D_1}^*]$, $G[V_{D_2}^*]$ are $\tau$-complete. By Lemma \ref{almost complete} (2), $G[V_{D_1}^*]$ contains $L_{\lfloor \frac{|V_{D_1}^*|}{2} \rfloor}$ and $G[V_{D_2}^*]$ contains $L_{\lfloor \frac{|V_{D_2}^*|}{2} \rfloor}$.

  Since $\delta < \frac{n}{2} - 1$, $n = 2\delta + K$ where $K \geq 3$.
  If $\delta$ is odd then there exists $i \in [l]$ such that $n_i > 2$.
  If $\delta$ is even, i.e, $4 | 2\delta$, and for any $i \in [l]$, $n_i = 2$ then $G$ contains disjoint cycles $C_{2n_1}, C_{2n_2}, \dots, C_{2n_l}$. 
  Indeed, if $|V^*_{D_1}| = 4t + b, |V^*_{D_2}| = 4t' + b'$ such that $b+b' > K, b,b' < 4$ then $b+b' \geq K+4 \geq 7$, so $b= 4$ or $b' =4$, a contradiction.
  Hence by Lemma \ref{weak ladder criteria - 1} and \ref{weak ladder of case r = 1}, it suffices to show that $G$ contains either $(\delta +2 , 2)$-weak ladder or $(\delta,1)$-weak ladder.
  Since $G$ is 2-connected, there is a matching of size two  in $G[V_{D_1}^*, V_{D_2}^*]$.

  \begin{claim} \label{attaching two components 1}
  If there exists a matching consisting of $\{u_1,u_2\}, \{v_1,v_2\} \in E(V_{D_1}^*, V_{D_2}^*)$ such that $N(u_1) \cap N(v_1) \cap V_{D_1}^* \neq \emptyset$ and $N(u_2) \cap N(v_2) \cap V_{D_2}^* \neq \emptyset$ 
  then $G[V_{D_1}^* \cup V_{D_2}^*]$ contains $(n',1)$-weak ladder where $n' \geq \lfloor \frac{|V_{D_1}^*| - 1}{2} \rfloor + \lfloor \frac{|V_{D_2}^*| - 1}{2} \rfloor.$
  \end{claim}
  \begin{proof}
  For $i \in [2]$, choose $z_i \in N(u_i) \cap N(v_i)$. 
  By Lemma \ref{almost complete} (4), $G[V_{D_i}^* \setminus \{u_i\}]$ contains $L_{\lfloor \frac{|V_{D_i}^*| - 1}{2} \rfloor}$ having $\{z_i,v_i\}$ as its first rung.
  By attaching these two ladders with $\{u_1,u_2\}, \{v_1,v_2\} $, we obtain a desired weak ladder.
  \end{proof}

  \begin{claim} \label{at least delta + 1}
  If $|V_{D_1}^*| \leq \delta$ then there exists a matching consisting of $\{u_1,u_2\}, \{v_1,v_2\} \in E(V_{D_1}^*, V_{D_2}^*)$ such that $u_1 \sim v_1$.
  \end{claim}
  \begin{proof}
  Let $I$ be a maximum independent set in $G[V_{D_1}^*]$. If $G[V_{D_1}^*]$ is complete then it is trivial, so we may assume $|I| \geq 2$.
  Choose $u_1 \in I, v_1 \in V_{D_1}^* \setminus I$ such that $u_1 \sim v_1$.
  Since $|I| \geq 2$, $|N(u_1) \cap V_{D_1}^*| \leq |V_{D_1}^*| - 2 \leq \delta - 2$, which implies that 
  $$ |N(u_1) \cap V_{D_2}^*| \geq 2.$$
  Since $|N(v_1) \cap V_{D_2}^*| \geq 1$, we can choose $v_2 \in N(v_1) \cap V_{D_2}^*$ and $u_2 \in N(u_1) \cap V_{D_2}^*$ such that $u_2 \neq v_2$.
  \end{proof}
  
  \begin{claim} \label{weak 1}
    If $|V_{D_1}^*|, |V_{D_2}^*| \leq \delta + 8$ then there exists a matching consisting of $\{u_1,u_2\} , \{v_1,v_2 \} \in E(V_{D_1}^*, V_{D_2}^*)$ such that one of the following holds.
    \begin{enumerate}
      \item $N(u_1) \cap N(v_1) \neq \emptyset$ and $N(u_2) \cap N(v_2) \neq \emptyset$.
      \item $u_1 \sim v_1$ or $u_2 \sim v_2$.
    \end{enumerate}
  \end{claim}
  \begin{proof}
  Suppose that for any two independent edges $\{u_1,u_2\} , \{v_1,v_2 \} \in E(V_{D_1}^*, V_{D_2}^*)$ , the first condition does not hold.
  Choose $\{u_1,u_2\} , \{v_1,v_2 \} \in E(V_{D_1}^*, V_{D_2}^*)$. 
  If $u_1 \sim v_1$ or $u_2 \sim v_2$ then the second condition holds, so we may assume that $u_1 \nsim v_1$ and $u_2 \nsim v_2$.
  Without loss of generality, $N(u_1) \cap N(v_1) = \emptyset$. Then $|N(u_1) \cap V_{D_1}^*| + |N(v_1) \cap V_{D_1}^*| \leq |V_{D_1}^*| - 2$, and then 
  $$|N(u_1) \cap V_{D_2}^*| + |N(v_1) \cap V_{D_2}^*| \geq 2\delta - (|V_{D_1}^*| - 2) \geq \delta - 6 \geq (1-\tau)|V_{D_2}^*|.$$
  Without loss of generality, $|N(v_1) \cap V_{D_2}^*| \geq |N(u_1) \cap V_{D_2}^*|,$ so $|N(v_1) \cap V_{D_2}^*| \geq  \frac{(1-\tau)|V_{D_2}^*|}{2}.$
  If $|N(u_1) \cap V_{D_2}^*| > \tau |V_{D_2}^*|$ then there exists $u_2' \in N(u_1) \cap V_{D_2}'$, 
  since $|N(u_2') \cap V_{D_2}'| \geq (1-\tau)|V_{D_2}'|$ there exists $v_2' \in N(u_2') \cap N(v_1) \cap V_{D_2}'$, so $\{u_1,v_1\}, \{u_2',v_2'\}$ are such that 2) holds.
  Otherwise, assume that $|N(u_1) \cap V_{D_2}^*| \leq \tau |V_{D_2}^*|$. 
  Note that $|N(u_2) \cap V_{D_2}'| \geq 4\tau |V_{D_2}^*|$. Since $|N(v_1) \cap V_{D_2}^*| \geq (1-2\tau)|V_{D_2}^*|$, 
  there exists $v_2' \in N(v_1) \cap N(u_2) \cap V_{D_2}^*$, so $\{u_1,v_1\}, \{u_2,v_2'\}$ are such that 2) holds.
  \end{proof}

  Without loss of generality $|V_{D_1}^*| \leq |V_{D_2}^*|$. 
  If $|V_{D_1}^*| \leq \delta$ then by Claim \ref{at least delta + 1} and Corollary \ref{attaching two components}, $G$ contains $(\delta,1)$-weak ladder.
  Hence we may assume that $|V_{D_1}^*| \geq \delta + 1.$
  If $|V_{D_2}^*| \geq \delta +9$ then $|V_{D_1}^*| + |V_{D_2}^*| \geq 2\delta +10$, and then by Corollary \ref{attaching two components}, 
  $G$ contains $(n',2)$-weak ladder where 
  $$ n' \geq \lfloor \frac{|V_{D_1}^*|}{2} \rfloor + \lfloor \frac{|V_{D_2}^*|}{2} \rfloor - 2 \geq \delta + 2. $$
  Hence $\delta +1 \leq |V_{D_1}^*| \leq |V_{D_2}^*| \leq \delta +8$.
  By Claim \ref{weak 1}, there exists a matching consisting of $\{u_1,u_2\} , \{v_1,v_2 \} \in E(V_{D_1}^*, V_{D_2}^*)$ such that one of the conditions from Claim \ref{weak 1} holds.
  If $N(u_1) \cap N(v_1) \neq \emptyset$ and $N(u_2) \cap N(v_2) \neq \emptyset$ then by Claim \ref{attaching two components 1}, $G$ contains 
  $(n',1)$-weak ladder where $n' \geq \lfloor \frac{|V_{D_1}^*| - 1}{2} \rfloor + \lfloor \frac{|V_{D_2}^*| - 1}{2} \rfloor \geq \delta.$
  Otherwise, $u_1 \sim v_1$ or $u_2 \sim v_2$, then by Corollary \ref{attaching two components}, $G$ contains $(n',1)$-weak ladder where $n' \geq \lfloor \frac{|V_{D_1}^*|}{2} \rfloor + \lfloor \frac{|V_{D_2}^*|}{2} \rfloor - 1 \geq \delta$.
  \end{proof}

  A similar, but more involved analysis can be done in the case when $\delta\geq n/2-1$, but additional obstructions occur.
  \begin{example}
  \begin{itemize}
  \item Suppose $n=4p+2$ for some odd integer $p$. Let $V, U$ be two disjoint sets of size $2p+1$ each and let
  $v_1,v_2\in V$ $u_1,u_2\in U$. Consider graph $G$ such that $G[V\setminus \{v_i\}]= K_{2p}$, $G[V\setminus \{u_i\}]= K_{2p}$ for $i=1,2$, and $v_iu_i\in G$ for $i=1,2$. Then $G$ is $2$-connected, $\delta(G)=2p$,
  but $G$ has no $p$ disjoint copies of $C_4$. Indeed, since $2p$ is not divisible by four, for $p$ disjoint copies of $C_4$ to exist, there would have a copy of $C_4$ with at least one vertex in each of $U$ and $V$.
  \item Suppose $n= 4p+1$ where $p\in \mathbb{Z}^+$ is odd. Consider $G$ obtained from two copies $K, K'$ of $K_{2p}$ by joining a vertex $v$ from $K$ with a vertex $v'$ from $K'$ by an edge and adding one more vertex $w\notin V(K)\cup V(K')$ and making it adjacent to $V(K) \cup V(K')\setminus \{v,v'\}$. Then $G$ is $2$-connected, $\delta(G)=2p$ but $G$ has no $p$ disjoint copies of $C_4$.
  Indeed, otherwise there would have to be copy of $C_4$ which has exactly two vertices in $K$ or $K'$. If there is a copy with exactly two vertices in $K$, then the remaining two must be $w$ and $v'$ which are not adjacent.
  \end{itemize}
  \end{example}
  \section{Extremal Case}\label{sec:extr}
  In this section we will prove the extremal case.
  \begin{theorem}\label{extremal case}
  Let $0< \alpha < \frac{1}{2}$ be given and $\beta$ be such that $\beta < (\frac{\alpha}{400})^2 \leq \frac{1}{640000}$.
  If $G$ is a graph on $n$ vertices with minimum degree $\delta\geq \alpha n$  which is $\beta$-extremal, then 
  either $G$ contains $L_{\delta}$ or $G$ is a subgraph of the graph from Example \ref{ex2}.
  Moreover, in the case when $G$ is a subgraph of the graph from Example \ref{ex2},
  for every $n_1, \dots, n_l \geq 2$ such that $\sum n_i = \delta$, $G$ contains disjoint cycles $C_{2n_1}, C_{2n_2}, \dots, C_{2n_l}$ if $n_i \geq 3$ for at least one $i$.
  \end{theorem}
  \begin{proof}
  Recall that $G$ is $\beta$-extremal if the exists a set $B\subset V(G)$ such that $|B| \geq (1-\delta/n-\beta)n$ and all but at most $4\beta n$ vertices $v\in B$ have $|N(v) \cap B| \leq \beta n$.
  Let $A = V(G) \backslash B$ and note that $\delta -\beta n \leq |A| \leq \delta +\beta n$, because for some $w\in B$, $|N(w)\cap A| \geq \delta -\beta n$.
  Let $C:= \{v\in B: |N(v) \cap B| > \beta n\}$, $A_1:=A, B_1:=B\setminus C$.
  Then $|B_1| \geq n-\delta -5\beta n$ and $|C|\leq 4{\beta}n$.
  Consequently, we have
  \begin{equation}\label{eq-ex1}
  |E(A_1, B_1)|\geq (\delta -5\beta n)|B_1|\geq (\delta -5\beta n)(n- \delta -5\beta n)\geq \delta n -\delta^2 -5\beta n^2.
  \end{equation}
  We have the following claim.
  \begin{claim} \label{first condition of extremal case}
  There are at most $ \sqrt{\beta}n$ vertices $v$ in $A_1$ such that $|N(v) \cap B_1| < n - \delta - 6\sqrt{\beta}n$.
  \end{claim}
  \begin{proof}
  Suppose not. Then
  \begin{align*}
  |E(A_1,B_1)| < (n- \delta - 6\sqrt{\beta}n) \cdot \sqrt{\beta} n + (n - \delta + \beta n)(\delta +\beta n - \sqrt{\beta}n) \leq \delta n-\delta^2 -5\beta n^2.
  \end{align*}
  This contradicts (\ref{eq-ex1}).
  \end{proof}
  
  Let $\gamma := 6\sqrt{\beta}$ and move those vertices $v\in A_1$ to $C$ for which $|N(v) \cap B_1| < n-\delta - 6\sqrt{\beta}n$.
  Let $A_2 := A_1\setminus C$, $B_2:=B_1$. Then, by Claim \ref{first condition of extremal case}, $|A_2| \geq \delta - (\beta + \sqrt{\beta})n$, $|B_2| \geq n- \delta - 5\beta n$ and,
  \begin{equation}
  |C| \leq (4\beta +\sqrt{\beta})n <2\sqrt{\beta}n.
  \end{equation}
  In addition, for every $v\in A_2$, $|N(v) \cap B_2| \geq n-\delta - 6\sqrt{\beta}n$
  and for every vertex $v\in B_2$, $|N(v) \cap A_2| \geq \delta  - (\beta + \sqrt{\beta})n$.
  
  We now partition $C= A_2'\cup B_2'$ as follows.
  Add $v$ to $A_2'$ if $|N(v) \cap B_2| \geq \gamma n$, and add it to $B_2'$ if  $|N(v) \cap A_2| \geq \gamma n$ and $\min \{|A_2\cup A_2'| , |B_2\cup B_2'|\} $ is as large as possible.
  Without loss of generality assume $|A_2\cup A_2'| \leq |B_2\cup B_2'|$. We have two cases.\\
  {\bf Case (i)} $|A_2\cup A_2'| \geq \delta$.\\
  For any $v \in A_2'$, $|N(v) \cap B_2| \geq \gamma n > |C|  \geq |A_2'|$. Therefore,
  there exists matching $M \in E(A_2', B_2)$ which saturates $A_2'$. Note that $q:=|M| = |A_2'| \leq |C| < 2\sqrt{\beta}n$. 
  For every $\{x_i, y_i\} \in M$, we can pick $x'_i,x''_i \in A_2, y'_i,y''_i \in B_2 \setminus V(M)$ all distinct, so that
  $\{x'_i,y'_i\}\in E, \{x''_i,y''_i\} \in E$ and $x'_i,x''_i \in N(y_i), y'_i,y''_i \in N(x_i)$.
  Note that this is possible because $|N(x_i) \cap B_2|  \geq 6\sqrt{\beta}n >3|M|$. Then $G[\{x_i,y_i,x'_i,x''_i,y'_i,y''_i \}]$ contains a 3-ladder, which we will denote by $L_i$.
  Note that $|\bigcup_{i\leq q} V(L_i)| = 3|M| < 6\sqrt{\beta}n$. We repeat the same process to find $p$ 3-ladders $L_j$ for each vertex from $B_2'$. We have $p+q = |C| < 2\sqrt{\beta}n$ 3-ladders, each containing exactly one vertex from $C$. 
  Note that $|A_2 \setminus \bigcup_{i=}^{p+q} V(L_i)|\leq |B_2 \setminus \bigcup_{i=1}^{p+q} V(L_i)|$.
  
  For every $v \in A_2\setminus \bigcup_{i=1}^{p+q} V(L_i)$,
  $|N(v) \cap (B_2 - (\bigcup_{i=1}^{p+q}V(L_i))) | \geq n-\delta -18\sqrt{\beta}n$ and for every
  $v\in B_2 \setminus \bigcup_{i=1}^{p+q} V(L_i)$, $|N(v)\cap A_2| \geq \delta - 18\sqrt{\beta}n$.
  Therefore there exists a matching $M'=\{\{a_i,b_i\} : i=1, \dots, |A_2 - \cup_{i=1}^{p+q} V(L_i)|\}$
  which saturates $A_2 - \bigcup_{i=1}^{p+q} V(L_i)$. Define the auxiliary graph $H$ as follows. For every $L_i$ consider vertex $V_{L_i}$ and let
  $$
  V(H) = \{v_{L_i} : i \in [p+q] \} \cup \{e : e \in M' \}.$$
  For
  $e=\{a_i,b_i\}, e'=\{a_j,b_j\}\in M'$, $\{e,e'\}\in E(H)$ if $G[\{a_i,a_j\},\{b_i,b_j\}]= K_{2,2}$ and for $v_{L_i}\in V(H), e=\{a_j,b_j\}\in M'$, $\{v_{L_i}, e\}\in E(H)$ if
  $a_j \in N(y'_i) \cap N(y''_i), b_j \in N(x'_i) \cap N(x''_i)$.
  Then $\delta(H) \geq |H| - 100\sqrt{\beta}n > \frac{|H|}{2}$, $H$ contains a Hamilton cycle, which gives, in turn, a $(|A_2|+|A_2'|)$-ladder in $G$.\\
  {\bf Case (ii)} $|A_2\cup A_2'| = \delta - K$ for some  $0< K \leq \beta + \sqrt{\beta}n < 2\sqrt{\beta}n$.\\
  Note that for every vertex $v\in B_2\cup B_2'$, $K \leq |N(v) \cap (B_2\cup B_2')|<(\gamma +2\sqrt{\beta})n$. Indeed, if $v\in B_2$, then $|N(v) \cap (B_2\cup B_2')|\leq \beta n +|B_2'|$ and if $v\in B_2'$, then  $|N(v) \cap (B_2\cup B_2')|< \gamma n +|B_2'|$ as otherwise we would move $v$ to $A_2'$. Thus, in particular,
  for every $v\in B_2\cup B_2'$, $|N(v)\cap A_2| \geq 9|A_2|/10$. In addition,
  $|A_2\cup A_2'| \leq |B_2\cup B_2'| -2K-2$. 
  
  Let $Q$ be a maximum triple matching in $G[B_2 \cup B_2']$ and $Q'$ be a maximum double matching in $G[B_2 \cup B_2' \setminus V(Q)]$.
  \begin{claim} \label{done with getting K triple matchings}
  If $|Q| + |Q'| \geq K$ and $|Q'| \leq 2$, then $G$ contains $L_{\delta}$.
  \end{claim}
  \begin{proof}
  Without loss of generality, let $|Q| = K - 2$ and $|Q'| = 2$.
  For $i \in [K-2]$, let $x_i$ denote the center of the $i$th star in the triple matching and let $x_i',y_i',z_i'$ be its leaves in $G[B_2 \cup B_2']$.
  Let $x_{K-1},x_K$ be the centers of the stars in the double matching and let $\{x_{K-1}', y_{K-1}'\}, \{x_K',y_K'\}$ denote the sets of leaves.
  Let $S:=\{x_1, \dots, x_K\}$ and note that $|S \cup A_2\cup A_2'|=\delta$. For every $z,w \in B_2 \cup B_2'$, $|N(w) \cap N(z) \cap A_2|\geq 4|A_2|/5$.
  Therefore, for any $i \in [K-2]$, there exists $y_i \in N(y_i') \cap N(x_i') \cap A_2$ and $z_i \in N(z_i') \cap N(x_i') \cap A_2$, i.e $G[\{x_i,x_i',y_i,y_i',z_i,z_i'\}]$ forms 3-ladder 
  and for $j \in \{K-1,K\}$, there exists $y_j \in N(x_j') \cap N(y_j') \cap A_2$, so $G[x_j,x_j',y_j,y_j']$ forms a 2-ladder, say $L_{j}$.
  As similar as we did in the case (i), we define auxiliary graph $H$ such that $V(H)$ consists of $K-2$ 3-ladders, $2$ 2-ladders and 3-ladders wrapping remaining vertices in $A_2' \cup B_2'$ and matchings in $E(A_2, B_2)$ saturating remaining vertices in $A_2$.
  For the definition of $E(H)$, only difference with what did in case (i) is for $v_{L_{K-1}}, v_{L_{K}}$. 
  For $e =\{a,b\} \in M $ ,$j \in \{K-1,K\}$, $\{v_{L_j}, e \} \in E(H)$ if $y_j \sim b $, $a \sim y_j'$. 
  Then $d_H(v_{L_{K-1}}), d_H(v_{L_{K}}) > \frac{|H|}{2}$ and $H - \setminus \{ v_{L_{K-1}} , v_{L_K} \}$ has a Hamilton cycle and then we obtain a Hamilton path in $H$ which has $v_{L_{K-1}}, v_{L_{K}}$ as its two ends.
  It implies that $G$ contains $L_{\delta}$.
  \end{proof}

  \begin{claim} \label{maximum matching in extremal case}
  If $K \geq 3$ then $|Q| \geq K$.
  \end{claim}
  \begin{proof}
  Suppose not. Then every vertex
  $v \in (B_2 \cup B_2') \backslash V(Q)$ has at least $K-2$ neighbors in $V(Q)$. Hence
  \begin{align*}
  |E(V(Q), (B_2\cup B_2') \backslash V(Q))| &\geq (K-2) | (B_2\cup B_2') \backslash V(Q)| \\
                        &> (K-2)(|(B_2\cup B_2')| - 4K) \\ 
                        &= (K-2)(n - \delta - 3K) \\ 
                        &> (K-2)( n -\delta - 6\sqrt{\beta}n) > Kn/5.
  \end{align*}
  Since for every $v\in B_2\cup B_2'$,  $|N(v) \cap (B_2\cup B_2') | < (\gamma +2\sqrt{\beta})n$,
  \begin{align*}
  |E(V(Q), (B_2\cup B_2') \backslash V(Q)) < 4K (\gamma +2\sqrt{\beta})n = 32K\sqrt{\beta}n.
  \end{align*}
  By combining these two inequalities, we obtain
  $$ \beta > (\frac{1}{160})^2,$$
  which is a contradiction to $\beta <\frac{1}{640000}$.
  \end{proof}
  
  By Claim \ref{maximum matching in extremal case} and \ref{done with getting K triple matchings}, we may assume that $K \leq 2$.
  Assume that $K=2$. By Claim \ref{done with getting K triple matchings}, $|Q| + |Q'| \leq 1$ and then every vertex $v \in  (B_2 \cup B_2') \backslash (V(Q) \cup V(Q'))$ has at least $K-1$ neighbors in $V(Q) \cup V(Q')$.
  By the same calculation as we did in Claim \ref{maximum matching in extremal case}, we will run into a contradiction.
  
  Finally, suppose that $K=1$, i.e, $|A_2 \cup A_2'| =\delta - 1$ and $\delta(G[B_2 \cup B_2']) \geq 1$. By Claim \ref{done with getting K triple matchings}, $\Delta(G[B_2 \cup B_2']) \leq 1$,
   which implies that $G[B_2 \cup B_2']$ is a perfect matching, so $|B_2 \cup B_2'|$ is even.
  Hence $G$ is a subgraph of the graph from Example \ref{ex2}. To prove the "Moreover" part, we proceed as follows.
  Let $\{v_1,v_1'\}, \{v_2, v_2'\} \in E(G[B_2 \cup B_2'])$. We have $|N(v_1)\cap N(v_2)\cap A_2|\geq 4|A_2|/5$ and
  $|N(v_1')\cap N(v_2')\cap A_2|\geq 4|A_2|/5$. Thus there is a copy of $C_6$ containing $\{v_1, v_1'\}, \{v_2, v_2'\}$,
  say $C_6 : x_1 v_1 v_1' x_1' v_2 v_2' x_1 $ where $x_1,x_1' \in A_2$.
  Similarly, $G[A_2 \cup A_2' \cup B_2 \cup B_2' \setminus V(C_6)]$ contains $L_{\delta -3}$ such that $z_1 \in N(x_1) \cap B_2, z_1' \in N(v_1) \cap A_2$ and $\{z_1, z_1'\}$ is the first rung of $L_{\delta -3}$.
  Let $n_l \geq 3$. Then the $C_6$ with first $n_l - 3$ rung contains $C_{2n_l}$ and remaining $L_{\delta - 3-(n_l - 3)} = L_{\delta - n_l}$ contains disjoint cycles $C_{2n_1},\dots, C_{2n_{l-1}}$.
  \end{proof}
  \section{Final comments}
  {\bf Proof of Corollary \ref{cor:Dense-case}.} Let $G$ be a graph on $n\geq N(\alpha/8)$ vertices such that $||G||\geq \alpha n^2$. If $\delta(G)\geq n/2$, then $G$ is pancyclic. If $n/2-1>\delta(G)\geq \alpha n/8$, then Theorem \ref{thm:Mainthm} implies that $G$ contains all even cycles of length $4, \dots, 2\delta(G)$.
  If $\delta(G) \geq n/2-1$, then $G$ contains all cycle if lengths $4, \dots, 2\delta(G)-2$ by Theorem \ref{thm:Mainthm} and a cycle on $2\delta(G)$ vertices by Theorem \ref{non extremal main}, Theorem \ref{extremal case} and Theorem \ref{non-extremal two}.
  
  Otherwise, by Mader's theorem, $G$ contains a subgraph $H$ which is $\alpha n/4$-connected. If $|H| \leq 2\delta(H)$ then $H$ is pancyclic by Bondy's theorem and we are done since $|H| \geq \alpha n/4\geq 2\delta(G)$. If $\delta(H) \leq |H|/2$, then by Theorem \ref{thm:Mainthm}, $H$ contains all even cycles $4, \dots, 2\delta(H)$ and $\delta(H)\geq \alpha n/4 \geq \delta(G).$ 
  $\Box$

  \end{document}